\documentclass[12pt]{amsart}
\usepackage{amsmath,amssymb,latexsym,esint,cite,mathrsfs}
\usepackage{verbatim,wasysym}
\usepackage[left=2.6cm,right=2.6cm,top=3cm,bottom=3cm]{geometry}
\makeatletter \@addtoreset{equation}{section} \makeatother

\usepackage{graphicx}
\usepackage{subfigure}
\usepackage{graphicx,epic,eepic}

\usepackage{color,enumitem,graphicx}
\usepackage[colorlinks=true,urlcolor=blue,
citecolor=red,linkcolor=blue,linktocpage,pdfpagelabels,
bookmarksnumbered,bookmarksopen]{hyperref}
\numberwithin{equation}{section}
\newtheorem{theorem}{Theorem}[section]
\newtheorem{lemma}[theorem]{Lemma}

\newtheorem{proposition}[theorem]{Proposition}
\newtheorem{corollary}[theorem]{Corollary}
\newtheorem{remark}[theorem]{Remark}
\numberwithin{equation}{section}

\allowdisplaybreaks

\begin{document}

\title[Some weighted fourth-order Hardy-H\'{e}non equations]
{Some weighted fourth-order Hardy-H\'{e}non equations}

\author[S. Deng]{Shengbing Deng$^{\ast}$}
\address{\noindent Shengbing Deng (Corresponding author) \newline
School of Mathematics and Statistics, Southwest University,
Chongqing 400715, People's Republic of China}\email{shbdeng@swu.edu.cn}

\author[X. Tian]{Xingliang Tian}
\address{\noindent Xingliang Tian  \newline
School of Mathematics and Statistics, Southwest University,
Chongqing 400715, People's Republic of China.}\email{xltian@email.swu.edu.cn}

\thanks{$^{\ast}$ Corresponding author}



\allowdisplaybreaks
%\maketitle

\begin{abstract}
{\tiny By using a suitable transform related to Sobolev inequality, we investigate the sharp constants and optimizers in radial space for the following weighted Caffarelli-Kohn-Nirenberg-type inequalities:
\begin{equation*}
\int_{\mathbb{R}^N}|x|^{\alpha}|\Delta u|^2 dx \geq S^{rad}(N,\alpha)\left(\int_{\mathbb{R}^N}|x|^{-\alpha}|u|^{p^*_{\alpha}} dx\right)^{\frac{2}{p^*_{\alpha}}}, \quad u\in C^\infty_c(\mathbb{R}^N),
\end{equation*}
where $N\geq 3$, $4-N<\alpha<2$, $p^*_{\alpha}=\frac{2(N-\alpha)}{N-4+\alpha}$.
Then we obtain the explicit form of the unique (up to scaling) radial positive solution $U_{\lambda,\alpha}$ to the weighted fourth-order Hardy (for $\alpha>0$) or H\'{e}non (for $\alpha<0$) equation:
\begin{equation*}
    \Delta(|x|^{\alpha}\Delta u)=|x|^{-\alpha} u^{p^*_{\alpha}-1},\quad u>0 \quad \mbox{in}\quad \mathbb{R}^N.
    \end{equation*}
For $\alpha\neq 0$, it is known the solutions of above equation are invariant for dilations $\lambda^{\frac{N-4+\alpha}{2}}u(\lambda x)$ but not for translations. However we show that if $\alpha$ is a negative even integer, there exist new solutions to the linearized problem, which related to above equation at $U_{1,\alpha}$, that ``replace'' the ones due to the translations invariance. This interesting phenomenon was first shown by Gladiali, Grossi and Neves [Adv. Math. 249, 2013, 1-36] for the second-order H\'{e}non problem.
    Finally, as applications, we investigate the remainder term of above inequality and also the existence of solutions to some related perturbed equations.
    }
\end{abstract}

\vspace{3mm}

\maketitle

\section{{\bfseries Introduction}}\label{sectir}

\subsection{Motivation} Recall the classical Sobolev inequality: for $N\geq 3$ there exists $S=S(N)>0$ such that
\begin{equation}\label{bsic}
\|\nabla u\|_{L^2(\mathbb{R}^N)}\geq S\|u\|_{L^{2^*}(\mathbb{R}^N)},\quad  \forall u\in D^{1,2}(\mathbb{R}^N),
\end{equation}
where $2^*=2N/(N-2)$ and $D^{1,2}(\mathbb{R}^N)$ denotes the closure of $C^\infty_c(\mathbb{R}^N)$ with respect to the norm $\|u\|_{D^{1,2}(\mathbb{R}^N)}=\|\nabla u\|_{L^2(\mathbb{R}^N)}$. It is well known that the Euler-Lagrange equation associated to (\ref{bsic}) is
\begin{equation}\label{bec}
-\Delta u=|u|^{2^*-2}u\quad \mbox{in}\quad \mathbb{R}^N.
\end{equation}
By Caffarelli et al. \cite{CGS89} and Gidas et al. \cite{GNN79}, it is known that all positive solutions are Talenti bubble \cite{Ta76} \[V_{z,\lambda}(x)=[N(N-2)]^{\frac{N-2}{4}}\left(\frac{\lambda}{1+\lambda^2|x-z|^2}\right)^{\frac{N-2}{2}},\]
with $z\in\mathbb{R}^N$ and $\lambda>0$.
%There are all minimizers of the Sobolev inequality for exponent two, up to scaling and translation.
The non-degeneracy of  $V_{z,\lambda}$ was given by Bianchi and Egnell \cite{BE91} (see also \cite[Lemma 3.1]{AGP99}), that is, the solutions of  the following linearized equation
\begin{equation}
-\Delta v=(2^*-1)V_{z,\lambda}^{2^*-2}v\quad \mbox{in}\quad \mathbb{R}^N,\quad v\in D^{1,2}(\mathbb{R}^N),
\end{equation}
are linear combinations of functions $\frac{\partial V_{z,\lambda}}{\partial \lambda}$ and $\frac{\partial V_{z,\lambda}}{\partial z_i}$, $i=1,\ldots,N$.

    In \cite{GGN13}, Gladiali, Grossi and Neves considered the second-order H\'{e}non equation
    \begin{equation}\label{Phs}
    -\Delta u=(N+l)(N-2)|x|^l u^{\frac{N+2+2l}{N-2}},\quad u>0 \quad \mbox{in}\quad \mathbb{R}^N,
    \end{equation}
    where $N\geq 3$ and $l>0$. This problem generalizes the well-known equation (\ref{bec}).
    Firstly, they gave the classification of radial solutions $V^\lambda_l$ in $D^{1,2}(\mathbb{R}^N)$ for problem (\ref{Phs}), where $V^\lambda_l(x)=\lambda^\frac{N-2}{2}V_l(\lambda x)$ and
    \[V_l(x)=(1+|x|^{2+l})^{-\frac{N-2}{2+l}}.\]
    Furthermore, they characterized all the solutions to the linearized problem related to (\ref{Phs}) at function $V_l$, that is
    \begin{equation}\label{Pwhls}
    -\Delta v=(N+l)(N+2+2l)|x|^l V_l^{\frac{N+2+2l}{N-2}-1}v \quad \mbox{in}\quad \mathbb{R}^N, \quad v\in D^{1,2}(\mathbb{R}^N).
    \end{equation}

\vskip0.25cm

    \noindent{\bf Theorem~A.} \cite[Theorem 1.3]{GGN13} {\it Let $l\geq 0$.  If $l>0$ is not an even integer, then the space of solutions of (\ref{Pwhls}) has dimension $1$ and is spanned by
    \begin{equation*}
    X_0(x)=\frac{1-|x|^{2+l}}{(1+|x|^{2+l})^\frac{N+l}{2+l}},
    \end{equation*}
    where $X_0\sim\frac{\partial V^\lambda_l}{\partial \lambda}|_{\lambda=1}$. If $l=2(k-1)$ for some $k\in\mathbb{N}^+$, then the space of solutions of (\ref{Pwhls}) has dimension $1+\frac{(N+2k-2)(N+k-3)!}{(N-2)!k!}$ and is spanned by
    \begin{equation*}
    X_0(x)=\frac{1-|x|^{2+l}}{(1+|x|^{2+l})^\frac{N+l}{2+l}},\quad X_{k,i}(x)=\frac{|x|^k\Psi_{k,i}(x)}{(1+|x|^{2+l})^\frac{N+l}{2+l}},
    \end{equation*}
    where $\{\Psi_{k,i}\}$, $i=1,\ldots,\frac{(N+2k-2)(N+k-3)!}{(N-2)!k!}$, form a basis of $\mathbb{Y}_k(\mathbb{R}^N)$, the space of all homogeneous harmonic polynomials of degree $k$ in $\mathbb{R}^N$.
    }

\vskip0.25cm

    Theorem A highlights the new phenomenon that if $l$ is an even integer then there exist new solutions to (\ref{Pwhls}) that ``replace'' the ones due to the translations invariance. It would be very interesting to understand if these new solutions are given by some geometrical invariants of the problem or not.
    It is obvious that for all $l>0$, $\frac{N+2+2l}{N-2}>2^*-1$ and the solution of equation (\ref{Phs}) is invariant for dilations but not for translations, since the presence of the term $|x|^l$ prevents the application of the moving plane method to obtain the radial symmetry of the solutions around some point in $\mathbb{R}^N$. Indeed, nonradial solutions appear. They constructed the nonradial solutions to equation (\ref{Phs}) when $l=2$ and $N\geq 4$ is even, that is, for any $a\in\mathbb{R}$, the functions
    \begin{equation*}
    u(x)=u(|x'|,|x''|)=(1+|x|^4-2a(|x'|^2-|x''|^2)+a^2)^{-\frac{N-2}{4}},
    \end{equation*}
    form a branch of solutions to (\ref{Phs}) bifurcating from $V_2$, where $(x',x'')\in \mathbb{R}^N=\mathbb{R}^{\frac{N}{2}}\times\mathbb{R}^{\frac{N}{2}}$.

    It is worth to mention that the authors of \cite{GGN13} came up with an interesting conjecture that the nonradial solutions exist only when $l$ is a positive even integer. Although they had no proof of this,  the classification result for a Liouville-type equation with singular data \cite{PT01}
    supports this conjecture. See \cite{BCG21,DGG17} for the Hardy-Sobolev equation of the similar  bifurcation phenomenon as in \cite{GGN13}.

    As applications of Theorem A, there are many results for the following second-order asymptotically critical H\'{e}non equation
    \begin{equation}\label{Pwhplh}
    -\Delta u=|x|^{l} |u|^{\frac{2(N+l)}{N-2}-2-\varepsilon}u \quad \mbox{in}\quad \Omega, \quad u=0\quad \mbox{on}\quad \partial\Omega,
    \end{equation}
    where $\Omega\subseteq \mathbb{R}^N$ is a smooth bounded domain containing the origin with $N\geq 3$, and $l>0$. Gladiali and Grossi in \cite{GG12} constructed a positive solution when $\varepsilon>0$ small enough and $0<l\leq 1$, this last bound on the exponent $l$ was removed in \cite{GGN13} getting the existence of positive solutions to equation (\ref{Pwhplh}) when $\varepsilon>0$ small enough and $l>0$ is not an even integer. Chen and Deng in \cite{CD17} constructed a sign-changing solution to (\ref{Pwhplh}) with the sharp of a tower of bubbles with alternate signs, centered at the origin when $\varepsilon\to 0^+$ and $l>0$ is not an even integer, see also \cite{CLP18} for the same result. If $l>0$ is an even integer, Alarc\'{o}n in \cite{Al18} gave further assumption about the domain that $\Omega$ is symmetric with respect to $x_1, x_2,\ldots, x_N$ and invariant for some suitable group, then constructed the same type sign-changing solutions to (\ref{Pwhplh}) as in \cite{CD17,CLP18}. For more results about the second-order H\'{e}non problem related to (\ref{Pwhplh}), readers can refer to \cite{EGPV21,FGP15,GG15,Liu20,Liu21}.

    Therefore, it is natural to consider the fourth-order Hardy or H\'{e}non problem and we hope to establish the analogous conclusion as
    \cite[Theorem 1.3]{GGN13}.

\subsection{Problem setup and main results}  Recently, Guo et al. \cite{GWW20} studied the weighted fourth-order elliptic equation:
    \begin{equation}\label{Pl}
    \Delta(|x|^{\alpha}\Delta u)=|x|^l u^p,\quad u\geq 0 \quad \mbox{in}\quad \mathbb{R}^N,
    \end{equation}
    where $N\geq 5$, $p>1$ and $4-N<\alpha<\min\{N,l+4\}$.
    Define
    \begin{equation}\label{defps}
    p_s:=\frac{N+4+2l-\alpha}{N-4+\alpha}.
    \end{equation}
    They obtained Liouville type result, that is, if $u\in C^4(\mathbb{R}^N\backslash\{0\}) \cap C^0(\mathbb{R}^N)$ with $|x|^{\alpha}\Delta u\in C^0(\mathbb{R}^N)$ is a nonnegative radial solution to (\ref{Pl}), then $u\equiv 0$ in $\mathbb{R}^N$ provided $1<p<p_s$. Successfully, Huang and Wang in \cite{HW20} gave the partial classifications of positive radial solutions for problem (\ref{Pl}) with $p=p_s$ and $l=-\alpha$, see also \cite{Ya21} for more general case.
    The method of \cite{HW20} is that making use of the transformation $v(t)=|x|^{\frac{N-4+\alpha}{2}}u(|x|)$, $t=-\ln |x|$, then changing problem (\ref{Pl}) to the following fourth-order ODE
    \begin{equation*}
    v^{(4)}-\frac{(N-2)^2+(2-\alpha)^2}{2}v''+\frac{(N-4+\alpha)^2(N-\alpha)^2}{16}v=v^{p}, \quad \mbox{in} \quad \mathbb{R}.
    \end{equation*}

    Equation (\ref{Pl}) is related to the H\'{e}non-Lane-Emden system
    \begin{eqnarray*}
    \left\{ \arraycolsep=1.5pt
       \begin{array}{ll}
        -\Delta u=|x|^{-\alpha}v^q, \quad &\mbox{in}\quad \mathbb{R}^N,\\[2mm]
        -\Delta v=|x|^{l}u^p,\quad &\mbox{in}\quad \mathbb{R}^N,
        \end{array}
    \right.
    \end{eqnarray*}
    with $q=1$. It is well-known that the following critical hyperbola plays an important role in existence results
    \begin{equation*}
    \frac{N-\alpha}{q+1}+\frac{N+l}{p+1}=N-2.
    \end{equation*}
    More precisely, Bidaut-Veron and Giacomini in \cite{BG10} have shown that if $N\geq 3$ and $\alpha, -l<2$, the above system admits a positive classical radial solution $(u,v)$ continuous at the origin if and only if $(p,q)$ is above or on the critical hyperbola. For more existence and non-existence results, refer to \cite{CH19,CM17,FG14,FKP21,Li21,Li98,Ph15} and the references therein.

    On the other hand, equation (\ref{Pl}) is closely related to Caffarelli-Kohn-Nirenberg-type (see \cite{CKN84} and we write (CKN) for short) inequalities
    \begin{equation*}
    \int_{\mathbb{R}^N}|x|^{\alpha}|\Delta u|^2 dx \geq C\left(\int_{\mathbb{R}^N}|x|^{l}|u|^{p} dx\right)^{\frac{2}{p}},
    \quad \mbox{for any}\quad u\in C^\infty_c(\mathbb{R}^N).
    \end{equation*}
    Inspired by \cite{GWW20}, and by using the (CKN) inequalities, we give a brief proof of the classification of positive radial solutions for problem (\ref{Pl}) with $l=-\alpha$ and $p=p_s$ which is different from \cite{HW20}.
     Firstly, we are mainly interested in a class of weighted higher-order (CKN) inequalities of the form
    \begin{equation}\label{Pi}
    \int_{\mathbb{R}^N}|x|^{\alpha}|\Delta u|^2 dx \geq S(N,\alpha)\left(\int_{\mathbb{R}^N}|x|^{-\alpha}|u|^{p^*_{\alpha}} dx\right)^{\frac{2}{p^*_{\alpha}}},
    \quad \mbox{for any}\quad u\in C^\infty_c(\mathbb{R}^N),
    \end{equation}
    for some positive constant $S(N,\alpha)$, where $N\geq 3$ and
    \begin{equation}\label{capc}
    4-N<\alpha<2,\quad p^*_{\alpha}:=\frac{2(N-\alpha)}{N-4+\alpha}.
    \end{equation}

    This problem generalizes the well-known high order Sobolev inequality
    \begin{equation}\label{bcesi}
    \int_{\mathbb{R}^N}|\Delta u|^2\geq S_2\left(\int_{\mathbb{R}^N}|u|^{\frac{2N}{N-4}}\right)^\frac{N-4}{N}
    \end{equation}
    for all $u\in D_0^{2,2}(\mathbb{R}^N)$ where $D_0^{2,2}(\mathbb{R}^N)=\{u\in L^{\frac{2N}{N-4}}(\mathbb{R}^N): \Delta u\in L^2(\mathbb{R}^N)\}$. The Euler-Lagrange equation associated to (\ref{bcesi}) is
    \begin{equation}\label{becbbb}
    \Delta^2 u=|u|^{\frac{8}{N-4}}u\quad \mbox{in}\quad \mathbb{R}^N.
    \end{equation}
    Smooth positive solutions to (\ref{becbbb}) have been completely classified in \cite{EFJ90}, where the authors proved that these solutions are given by
    \begin{equation*}
    W_{z,\lambda}(x)=[(N-4)(N-2)N(N+2)]^{\frac{N-4}{8}}\left(\frac{\lambda}{1+\lambda^2|x-z|^2}\right)^{\frac{N-4}{2}},
    \end{equation*}
    with $\lambda>0$ and $z\in\mathbb{R}^N$ and they are extremal functions for (\ref{bcesi}).

    Coming back to (\ref{Pi}), we define $D^{2,2}_\alpha(\mathbb{R}^N)$ as the completion of $C^\infty_c(\mathbb{R}^N)$ with the inner product
    \begin{equation}\label{defd22i}
    \langle\phi,\varphi\rangle_\alpha=\int_{\mathbb{R}^N}|x|^{\alpha}\Delta \phi\Delta \varphi dx,
    \end{equation}
    and the norm $\|\phi\|_{D^{2,2}_\alpha(\mathbb{R}^N)}=\langle\phi,\phi\rangle^{1/2}_\alpha$.
    Define also $L^{p^*_{\alpha}}_{\alpha}(\mathbb{R}^N)$ the space of functions $\phi$ such that $\int_{\mathbb{R}^N}|x|^{-\alpha}|\phi|^{p^*_{\alpha}} dx<\infty$ with the norm $\|\phi\|_{L^{p^*_{\alpha}}_{\alpha}(\mathbb{R}^N)}=(\int_{\mathbb{R}^N}|x|^{-\alpha}|\phi|^{p^*_{\alpha}} dx)^{1/p^*_{\alpha}}$. Therefore, (\ref{Pi}) can be stated as that the embedding $D^{2,2}_\alpha(\mathbb{R}^N)\hookrightarrow L^{p^*_{\alpha}}_{\alpha}(\mathbb{R}^N)$ is continuous.
    The best constant in (\ref{Pi}) is given by
    \begin{equation}\label{defbcsg}
    S(N,\alpha)=\inf_{u\in D^{2,2}_\alpha(\mathbb{R}^N)\backslash\{0\}}\frac{\int_{\mathbb{R}^N}|x|^{\alpha}|\Delta u|^2 dx}{\left(\int_{\mathbb{R}^N}|x|^{-\alpha}|u|^{p^*_{\alpha}} dx\right)^{\frac{2}{p^*_{\alpha}}}}.
    \end{equation}
    In this paper, we just consider the radial extremal functions to (CKN) inequality, so we define
    \[
    D^{2,2}_{\alpha,rad}(\mathbb{R}^N):=\{u\ :\ u(x)=u(|x|), u\in D^{2,2}_\alpha(\mathbb{R}^N)\}
    \]
    and
    \begin{equation}\label{defbcs}
    S^{rad}(N,\alpha):=\inf_{u\in D^{2,2}_{\alpha,rad}(\mathbb{R}^N)\backslash\{0\}}\frac{\int_{\mathbb{R}^N}|x|^{\alpha}|\Delta u|^2 dx}{\left(\int_{\mathbb{R}^N}|x|^{-\alpha}|u|^{p^*_{\alpha}} dx\right)^{\frac{2}{p^*_{\alpha}}}}.
    \end{equation}
    We will give the explicit forms for all maximizers and the exact best constant for $S^{rad}(N,\alpha)$ as the following:

    \begin{theorem}\label{thmbcm}
    Assume that $N\geq 3$, $4-N<\alpha<2$. We have
    \begin{equation*}
    S^{rad}(N,\alpha)=\left(\frac{2-\alpha}{2}\right)^{\frac{4N-4-2\alpha}{N-4}}
    \left(\frac{2\pi^{\frac{N}{2}}}{\Gamma(\frac{N}{2})}\right)^{\frac{4-2\alpha}{N-\alpha}}C\left(\frac{2N-2\alpha}{2-\alpha}\right),
    \end{equation*}
    where $C(M)=(M-4)(M-2)M(M+2)\left[\Gamma^2(\frac{M}{2})/(2\Gamma(M))\right]^{\frac{4}{M}}$. % is the best constant of embedding $D^{2,2}(\mathbb{R}^M)\hookrightarrow L^{\frac{2M}{M-4}}(\mathbb{R}^M)$ for $M>4$.
    Moreover the extremal functions which achieve $S^{rad}(N,\alpha)$ in (\ref{defbcs}) are unique (up to scaling) and given by
    \begin{equation}\label{bcm}
    V_{\lambda,\alpha}(x)=\frac{A\lambda^{\frac{N-4+\alpha}{2}}}{(1+\lambda^{2-\alpha}|x|^{2-\alpha})^{\frac{N-4+\alpha}{2-\alpha}}},
    \end{equation}
    for all $A\in\mathbb{R}\backslash\{0\}$ and $\lambda>0$.
    \end{theorem}

    It is well-known that Euler-Lagrange equation of (\ref{Pi}), up to scaling,  is given by %as in (\ref{Pl}) with $l=-\alpha$ and $p=p_s$, that is,
    \begin{equation}\label{Pwh}
    \Delta(|x|^{\alpha}\Delta u)=|x|^{-\alpha} |u|^{p^*_{\alpha}-2}u,\quad \mbox{in}\quad \mathbb{R}^N.
    \end{equation}
    Therefore, as the direct consequence of Theorem \ref{thmbcm}, we obtain

    \begin{corollary}\label{thmpwh}
    Assume that $N\geq 3$, $4-N<\alpha<2$. Then equation (\ref{Pwh}) has a unique (up to scaling) positive radial solution of the form
    \begin{equation*}
    U_{\lambda,\alpha}(x)=\frac{C_{N,\alpha}\lambda^{\frac{N-4+\alpha}{2}}}{(1+\lambda^{2-\alpha}|x|^{2-\alpha})^\frac{N-4+\alpha}{2-\alpha}},
    \end{equation*}
    with $\lambda>0$, where $C_{N,\alpha}=\left[(N-4+\alpha)(N-2)(N-\alpha)(N+2-2\alpha)\right]^{\frac{N-4+\alpha}{8-4\alpha}}$.
    \end{corollary}

    Inspired by \cite{GGN13}, then we concern the linearized problem related to (\ref{Pwh}) at the function $U_{1,\alpha}$. This leads to study the problem
    \begin{equation}\label{Pwhl}
    \Delta(|x|^{\alpha}\Delta v)=(p^*_{\alpha}-1)|x|^{-\alpha} U_{1,\alpha}^{p^*_{\alpha}-2}v \quad \mbox{in}\quad \mathbb{R}^N, \quad v\in D^{2,2}_\alpha(\mathbb{R}^N).
    \end{equation}
    Next theorem characterizes all the solutions to (\ref{Pwhl}).

    \begin{theorem}\label{thmpwhl}
    Assume that $N\geq 3$, $4-N<\alpha<2$. If $\alpha$ is not a negative even integer, then the space of solutions of (\ref{Pwhl}) has dimension $1$ and is spanned by
    \begin{equation*}
    Z_0(x)=\frac{1-|x|^{2-\alpha}}{(1+|x|^{2-\alpha})^\frac{N-2}{2-\alpha}},
    \end{equation*}
    where $Z_0\sim\frac{\partial U_{\lambda,\alpha}}{\partial \lambda}|_{\lambda=1}$, and in this case we say $U_{1,\alpha}$ is non-degenerate. Otherwise, if $\alpha=-2(k-1)$ for some $k\in\mathbb{N}^+$, then the space of solutions of (\ref{Pwhl}) has dimension $1+\frac{(N+2k-2)(N+k-3)!}{(N-2)!k!}$ and is spanned by
    \begin{equation*}
    Z_0(x)=\frac{1-|x|^{2-\alpha}}{(1+|x|^{2-\alpha})^\frac{N-2}{2-\alpha}},\quad Z_{k,i}(x)=\frac{|x|^k\Psi_{k,i}(x)}{(1+|x|^{2-\alpha})^\frac{N-2}{2-\alpha}},
    \end{equation*}
    where $\{\Psi_{k,i}\}$, $i=1,\ldots,\frac{(N+2k-2)(N+k-3)!}{(N-2)!k!}$, form a basis of $\mathbb{Y}_k(\mathbb{R}^N)$, the space of all homogeneous harmonic polynomials of degree $k$ in $\mathbb{R}^N$.
    \end{theorem}

    \begin{remark}\label{rem:exf}\rm
    The key step of the proofs for Theorems \ref{thmbcm} and \ref{thmpwhl} is the change of variable $r\mapsto r^{\frac{2}{2-\alpha}}$, i.e. we set $v(s)=u(r)$ and $r=s^{\frac{2}{2-\alpha}}$, which was used in \cite{CG10} in a different context, see also \cite[Theorem A.1]{GGN13}. It is a surprising thing that we only need to suppose $N\geq 3$ when dealing with the weighted fourth-order Hardy-H\'{e}non equation. Indeed, when we deal with the minimizers for $S^{rad}(N,\alpha)$ in (\ref{defbcs}), the only fact we have used is that %$D^{2,2}(\mathbb{R}^M)\hookrightarrow L^{\frac{2M}{M-4}}(\mathbb{R}^M)$ for $M=\frac{2N-2\alpha}{2-\alpha}>4$, i.e. $4-N<\alpha<2$.
    \begin{equation*}
    \int^\infty_0\left[v''(s)+\frac{M-1}{s}v'(s)\right]^2 s^{M-1}ds
    \geq C(M)\left(\int^\infty_0|v(s)|^{\frac{2M}{M-4}}s^{M-1}ds\right)^{\frac{M-4}{M}},
    \end{equation*}
    for all $v\in C^2_c(\mathbb{R})\backslash\{0\}$ satisfying $\int^\infty_0\left[v''(s)+\frac{M-1}{s}v'(s)\right]^2 s^{M-1}ds<\infty$, where $M=\frac{2N-2\alpha}{2-\alpha}>4$, i.e. $4-N<\alpha<2$ which requires $N>2$, or $2<\alpha<4-N$ which shows $N=1$ and $\alpha\in(2,3)$. If $N=1$, then $\alpha\in(2,3)$ indicates that $\alpha$ can not be a negative even integer, thus we only deal with the case $N\geq 3$ and $4-N<\alpha<2$.

    It is worth to mention that $S(N,\alpha)\leq S^{rad}(N,\alpha)$, and $S(N,\alpha)$ in (\ref{defbcsg}) might be zero for some special $\alpha$ (see \cite{CaM11}). Furthermore, when $S(N,\alpha)>0$ it may also be achieved by non-radial functions and thus (\ref{Pwh}) might exist non-radial positive solutions. In fact, let $\alpha=-2$, $N\geq 8$ be even and $\mathbb{R}^{N}=\mathbb{R}^{\frac{N}{2}}\times \mathbb{R}^{\frac{N}{2}}$, $x=(x',x'')$ with $x'\in\mathbb{R}^{\frac{N}{2}}$ and $x''\in\mathbb{R}^{\frac{N}{2}}$, then for any $a\in\mathbb{R}$ the functions
    \begin{equation}\label{defbrs}
    v(x)=v(|x'|,|x''|)=C_{N,-2}(1+|x|^4-2a(|x'|^2-|x''|^2)+a^2)^{-\frac{N-6}{4}},
    \end{equation}
    form a branch of solutions to (\ref{Pwh}) bifurcating from $U_{1,-2}$.
    \end{remark}

    From Theorem \ref{thmpwhl}, we know that $U_{1,\alpha}$ is non-degenerate when $\alpha$ is not a negative even integer. By this result, we can consider several simple applications of Theorem \ref{thmpwhl}. Enlightened by Brezis and Lieb \cite{BrE85}, the first thing we care about is the remainder term of (CKN) inequality (\ref{Pi}) in radial space $D^{2,2}_{\alpha,rad}(\mathbb{R}^N)$. The Sobolev inequality states that there exists constant $\mathcal{S}$ depending only on $N$ and $s$ such that
    \begin{equation}\label{bsics}
    \|u\|^2_{D^{s,2}_0(\mathbb{R}^N)}\geq \mathcal{S}\|u\|^2_{L^{\frac{2N}{N-2s}}(\mathbb{R}^N)},\quad \mbox{for all}\quad u\in D_0^{s,2}(\mathbb{R}^N),
    \end{equation}
    where $0<s<N/2$ and $D_0^{s,2}(\mathbb{R}^N)$ is the space of all tempered distributions $u$ such that
     \[\widehat{u}\in L^1_{loc}(\mathbb{R}^N)\quad \mbox{and}\quad \|u\|^2_{D^{s,2}_0(\mathbb{R}^N)}:=\int_{\mathbb{R}^N}|\xi|^s|\widehat{u}|^2<\infty.\]
    Here, as usual, $\widehat{u}$ denotes the (distributional) Fourier transform of $u$.
      It is well known that the extremal functions of best constant $\mathcal{S}$ are given as the set functions which, up to translation, dilation and multiplication by a nonzero constant, coincide with $W(x)=(1+|x|^2)^{-(N-2s)/2}$.
    For $s=1$, Brezis and Lieb \cite{BrE85} asked the question whether a remainder term - proportional to the quadratic distance of the function $u$ to be the manifold $\mathcal{M}:=\{c\lambda^{(N-2s)/2}W(\lambda(x-z): z\in\mathbb{R}^N, c\in\mathbb{R}, \lambda>0\}$ - can be added to the right hand side of (\ref{bsics}). This question was answered affirmatively in the case $s=1$ by Bianchi and Egnell \cite{BE91}, and their result was extended later to the case $s=2$ by Lu and Wei \cite{LW00} and the the case of an arbitrary even positive integer $N>2s$ in \cite{BWW03}, and the whole interval case $s\in (0,N)$ was proved in \cite{CFW13}.
    Furthermore, R\u{a}dulescu et. al \cite{RSW02} gave the remainder term of Hardy-Sobolev inequality for exponent two.
    Wang and Willem \cite{wangwil} studied Caffarelli-Kohn-Nirenberg inequalities with remainder terms.
    Recently, Wei and Wu \cite{WW22} established the stability of the profile decompositions to a special case of the (CKN) inequality and also gave the remainder term.

    As mentioned above, it is natural to establish (CKN) inequality (\ref{Pi}) with remainder terms in the radial space $D^{2,2}_{\alpha,rad}(\mathbb{R}^N)$ under the help of Theorem \ref{thmpwhl} when $\alpha$ is not a negative even integer, as an analogous result to \cite{LW00}.

    \begin{theorem}\label{thmprt}
    Assume $N\geq 3$, and let $4-N<\alpha<2$ be not a negative even integer. Then there exists constant $B=B(N,\alpha)>0$ such that for every $u\in D^{2,2}_{\alpha,rad}(\mathbb{R}^N)$, it holds that
    \[
    \int_{\mathbb{R}^N}|x|^{\alpha}|\Delta u|^2 dx-S^{rad}(N,\alpha)\left(\int_{\mathbb{R}^N}|x|^{-\alpha}|u|^{p^*_{\alpha}} dx\right)^{\frac{2}{p^*_{\alpha}}}
    \geq B {\rm dist}(u,\mathcal{M}_2)^2,
    \]
    where $\mathcal{M}_2=\{cU_{\lambda,\alpha}: c\in\mathbb{R}, \lambda>0\}$ is a two-dimensional manifold, and ${\rm dist}(u,\mathcal{M}_2):=\inf_{\phi\in \mathcal{M}_2}\|\phi-u\|_{D^{2,2}_\alpha(\mathbb{R}^N)}=\inf_{c\in\mathbb{R}, \lambda>0}\|u-cU_{\lambda,\alpha}\|_{D^{2,2}_\alpha(\mathbb{R}^N)}$.
    \end{theorem}

    The second thing we want to study is to construct solutions by using the Lyapunove-Schmidt argument, enlightened by \cite{AGP99}
    (and also \cite[Sections 3 and 4]{FS03}). Now, we will establish sufficient conditions on a prescribed weighted $h(x)$ on $\mathbb{R}^N$ which guarantee the existence of solutions to the perturbative model problem
    \begin{equation}\label{Pwhp}
    \Delta(|x|^{\alpha}\Delta u)=(1+\varepsilon h(x))|x|^{-\alpha} u^{p^*_{\alpha}-1},\quad u>0 \quad \mbox{in}\quad \mathbb{R}^N, \quad u\in D^{2,2}_{\alpha,rad}(\mathbb{R}^N).
    \end{equation}

    \begin{theorem}\label{thmpwhp}
    Assume $N\geq 3$, and let $4-N<\alpha<2$ be not a negative even integer, $h\in L^\infty(\mathbb{R}^N)\cap C(\mathbb{R}^N)$. If $\lim_{|x|\to 0}h(x)=\lim_{|x|\to \infty}h(x)=0$, then equation (\ref{Pwhp}) has at least one solution for any $\varepsilon$ close to zero.
    \end{theorem}

    The paper is organized as follows: In Section \ref{sectpmr} we deduce the optimizers of (CKN) inequality and characterize all solutions to the linearized Hardy-H\'{e}non equation (\ref{Pwhl}). In Section \ref{sect:rt}, we study the remainder term of (CKN) inequality (\ref{Pi}) and prove Theorem \ref{thmprt}. In Section \ref{sectprp} we investigate the existence of solutions to the related perturbed equation (\ref{Pwhp}) by using finite dimensional Lyapunov-Schmit reduction method and prove Theorem \ref{thmpwhp}.

\section{{\bfseries Optimizers of (CKN) inequality and linearized problem}}\label{sectpmr}

    In this section, at first, we use a suitable transform that is changing the variable $r\mapsto r^{\frac{2}{2-\alpha}}$, related to Sobolev inequality to investigate the sharp constants and optimizers of (CKN) inequality (\ref{Pi}) in radial space $D^{2,2}_{\alpha,rad}(\mathbb{R}^N)$.

\subsection{ Proof of Theorem \ref{thmbcm}.} We follow the arguments in the proof of \cite[Theorem A.1]{GGN13}. Let $u\in D^{2,2}_{\alpha,rad}(\mathbb{R}^N)$. Making the changes that $v(s)=u(r)$ and $r=s^q$ where $q>0$ will be given later, then we have
    \begin{equation*}
    \begin{split}
    \int_{\mathbb{R}^N}|x|^{\alpha}|\Delta u|^2 dx
    = & \omega_{N-1}\int^\infty_0 r^\alpha\left[u''(r)+\frac{N-1}{r}u'(r)\right]^2 r^{N-1}dr \\
    = & \omega_{N-1} q^{-3}\int^\infty_0\left[v''(s)+\frac{(N-1)q-(q-1)}{s}v'(s)\right]^2 s^{(N-1)q-3(q-1)+q \alpha}ds,
    \end{split}
    \end{equation*}
    where $\omega_{N-1}$ is the surface area for unit ball of $\mathbb{R}^N$.
    In order to make use of Sobolev inequality, we need $(N-1)q-(q-1)=(N-1)q-3(q-1)+q \alpha$ which requires \[q=\frac{2}{2-\alpha}.\] Now, we set
    \begin{equation}\label{defm}
    M:=(N-1)q-(q-1)+1=\frac{2(N-\alpha)}{2-\alpha}>4,
    \end{equation}
    which implies
    \begin{equation*}
    \begin{split}
    \int^\infty_0 r^\alpha\left[u''(r)+\frac{N-1}{r}u'(r)\right]^2 r^{N-1}dr
    = & q^{-3}\int^\infty_0\left[v''(s)+\frac{M-1}{s}v'(s)\right]^2 s^{M-1}ds.
    \end{split}
    \end{equation*}

    When $M$ is an integer we can use the classical Sobolev inequality (see \cite{Li85-1,Li85-2}) and we get
    \begin{equation*}
    \begin{split}
    \int^\infty_0\left[v''(s)+\frac{M-1}{s}v'(s)\right]^2 s^{M-1}ds
    \geq & C(M)\left(\int^\infty_0|v(s)|^{\frac{2M}{M-4}}s^{M-1}ds\right)^{\frac{M-4}{M}} \\
    = & q^{-\frac{M-4}{M}}C(M)\left(\int^\infty_0|u(r)|^{\frac{2M}{M-4}}r^{\frac{M}{q}-1}dr\right)^{\frac{M-4}{M}},
    \end{split}
    \end{equation*}
    where $C(M)=\pi^2(M+2)M(M-2)(M-4)\left(\Gamma(M/2)/\Gamma(M)\right)^{\frac{4}{M}}\left(2\pi^{M/2}/\Gamma(M/2)\right)^{-\frac{4}{M}}$ (see \cite[(1.4)]{Va93}). Moreover, even $M$ is not an integer we readily see that the above inequality remains true.

    From (\ref{defm}), we deduce that \[\frac{2M}{M-4}=\frac{2(N-\alpha)}{N-4+\alpha}=p^*_{\alpha},\quad \frac{M}{q}-1=N-1-\alpha.\] So we get
    \begin{equation*}
    \begin{split}
    & \int^\infty_0 r^\alpha\left[u''(r)+\frac{N-1}{s}u'(r)\right]^2 r^{N-1}dr\geq q^{-3-\frac{M-4}{M}}C(M)\left(\int^\infty_0r^{-\alpha}|u(r)|^{p^*_{\alpha}}r^{N-1}dr\right)^{\frac{2}{p^*_{\alpha}}},
    \end{split}
    \end{equation*}
    which proves (\ref{defbcs}) with
    \begin{equation*}
    \begin{split}
    S^{rad}(N,\alpha)
    = & q^{-3-\frac{M-4}{M}}\omega^{1-\frac{2}{p^*_{\alpha}}}_{N-1}C(M)
    =\left(\frac{2-\alpha}{2}\right)^{\frac{4N-4-2\alpha}{N-4}}
    \left(\frac{2\pi^{\frac{N}{2}}}{\Gamma(\frac{N}{2})}\right)^{\frac{4-2\alpha}{N-\alpha}}C\left(\frac{2N-2\alpha}{2-\alpha}\right).
    \end{split}
    \end{equation*}

    Moreover, from the previous inequalities, we also get that the extremal functions are obtained as
    \begin{equation*}
    \begin{split}
    \int^\infty_0\left[v_\nu''(s)+\frac{M-1}{s}v_\nu'(s)\right]^2 s^{M-1}ds=C(M)\left(\int^\infty_0|v_\nu(s)|^{\frac{2M}{M-4}}s^{M-1}ds\right)^{\frac{M-4}{M}}.
    \end{split}
    \end{equation*}
    It is well known that \[v_\nu(s)=A\nu^{\frac{M-4}{2}}(1+\nu^2s^2)^{-\frac{M-4}{2}}\] for all $A\in\mathbb{R}$ and $\nu\in\mathbb{R}^+$, see \cite[Theorem 2.1]{EFJ90}. Setting $\nu=\lambda^{1/q}$ and $s=|x|^{1/q}$, then we obtain that all the radial extremal functions of $S^{rad}(N,\alpha)$ have the form
    \begin{equation}\label{defula}
    V_{\lambda,\alpha}(x)=\frac{A\lambda^{\frac{N-4+\alpha}{2}}}{(1+\lambda^{2-\alpha}|x|^{2-\alpha})^{\frac{N-4+\alpha}{2-\alpha}}},
    \end{equation}
    for all $A\in\mathbb{R}$ and $\lambda>0$. The proof of Theorem \ref{thmbcm} is now complete.

    \qed

    Now, we are going to show the uniqueness of positive radial solutions of equation (\ref{Pwh}).

    Let $u(x)\in \mathcal{\mathcal{D}}^{2,2}_{\alpha,rad}(\mathbb{R}^N)$ be a positive radial solution of equation (\ref{Pwh}) and $X(s)=u(r)$ where $|x|=r=s^{q}$ and $q=2/(2-\alpha)$, then by simple calculation,  (\ref{Pwh}) is equivalent to
    \begin{equation}\label{PpwhlWe}
    \begin{split}
    & X^{(4)}(s)+\frac{2(M-1)}{s}X'''(s)+\frac{(M-1)(M-3)}{s^2}X''(s)-\frac{(M-1)(M-3)}{s^3}X'(s) \\
    = & q^{4}|X|^{\frac{8}{M-4}}X,\quad \mbox{in}\quad s\in(0,\infty)
    \end{split}
    \end{equation}
    where $M=\frac{2(N-\alpha)}{2-\alpha}>4$, since $p^*_{\alpha}=\frac{2M}{M-4}$. Then from \cite[Theorem 1.3]{Li98}, we know that equation (\ref{PpwhlWe}) has a unique (up to scalings) positive solution of the form
    \begin{equation}\label{eqsfe}
    X(s)=\frac{C_{M,q}\nu^{\frac{M-4}{2}}}{(1+\nu^{2}s^{2})^{\frac{M-4}{2}}},
    \end{equation}
    for some constant $\nu>0$, where $C_{M,q}=\left[q^{-4}(M-4)(M-2)M(M+2)\right]^{\frac{M-4}{8}}$. That is, equation (\ref{Pwh}) has a unique (up to scalings)
    radial solution of the form
    \begin{equation}\label{defulae}
    u(x)=\frac{C_{N,\alpha}\lambda^{\frac{N-4+\alpha}{2}}}{(1+\lambda^{2-\alpha}|x|^{2-\alpha})^\frac{N-4+\alpha}{2-\alpha}},
    \end{equation}
    for some $\lambda>0$, where $C_{N,\alpha}=\left[(N-4+\alpha)(N-2)(N-\alpha)(N+2-2\alpha)\right]^{\frac{N-4+\alpha}{8-4\alpha}}$. Therefore, Corollary \ref{thmpwh} holds.

    Then by using the standard spherical decomposition and taking the changes of variable $r\mapsto r^{\frac{2}{2-\alpha}}$, we can characterize all solutions to the linearized problem (\ref{Pwhl}).

\subsection{Proof of Theorem \ref{thmpwhl}.} We follow the arguments in the proof of \cite[Theorem 1.3]{GGN13}, and also \cite[Theorem 2.2]{LW00}. Equation (\ref{Pwhl}) is equivalent to
    \begin{equation}\label{Pwhlp}
    \Delta(|x|^{\alpha}\Delta v)=\frac{(p^*_{\alpha}-1)C_{N,\alpha}^{p^*_{\alpha}-2}}{|x|^{\alpha} (1+|x|^{2-\alpha})^4}v \quad \mbox{in}\quad \mathbb{R}^N, \quad v\in D^{2,2}_\alpha(\mathbb{R}^N).
    \end{equation}
    %here $S_{N,\alpha}=(p^*_{\alpha}-1)C_{N,\alpha}^{p^*_{\alpha}-2}$.
    We will decompose the fourth-order equation (\ref{Pwhlp}) into a system of two second-order equations. Firstly, we decompose $v$ as follows:
    \begin{equation}\label{defvd}
    v(r,\theta)=\sum^{\infty}_{k=0}\phi_k(r)\Psi_k(\theta),\quad \mbox{where}\quad r=|x|,\quad \theta=\frac{x}{|x|}\in \mathbb{S}^{N-1},
    \end{equation}
    and
    \begin{equation*}
    \phi_k(r)=\int_{\mathbb{S}^{N-1}}v(r,\theta)\Psi_k(\theta)d\theta.
    \end{equation*}
    Here $\Psi_k(\theta)$ denotes the $k$-th spherical harmonic, i.e., it satisfies
    \begin{equation}\label{deflk}
    -\Delta_{\mathbb{S}^{N-1}}\Psi_k=\lambda_k \Psi_k,
    \end{equation}
    where $\Delta_{\mathbb{S}^{N-1}}$ is the Laplace-Beltrami operator on $\mathbb{S}^{N-1}$ with the standard metric and  $\lambda_k$ is the $k$-th eigenvalue of $-\Delta_{\mathbb{S}^{N-1}}$. It is well known that $\lambda_k=k(N-2+k)$, $k=0,1,2,\ldots$ whose multiplicity is \[\frac{(N+2k-2)(N+k-3)!}{(N-2)!k!}\] and that \[{\rm Ker}(\Delta_{\mathbb{S}^{N-1}}+\lambda_k)=\mathbb{Y}_k(\mathbb{R}^N)|_{\mathbb{S}^{N-1}},\] where $\mathbb{Y}_k(\mathbb{R}^N)$ is the space of all homogeneous harmonic polynomials of degree $k$ in $\mathbb{R}^N$. It is standard that $\lambda_0=0$ and the corresponding eigenfunction of (\ref{deflk}) is the constant function. The second eigenvalue $\lambda_1=N-1$ and the corresponding eigenfunctions of (\ref{deflk}) are $\frac{x_i}{|x|}$, $i=1,\ldots,N$.

    Moreover, let
    \begin{equation*}
    \psi_k(r)=-\int_{\mathbb{S}^{N-1}}|x|^\alpha\Delta v(r,\theta)\Psi_k(\theta) d\theta, \quad \mbox{i.e.,}\quad r^{-\alpha}\psi_k(r)=-\int_{\mathbb{S}^{N-1}}\Delta v(r,\theta)\Psi_k(\theta)d\theta.
    \end{equation*}
    It is known that
    \begin{equation}\label{Ppwhl2deflklw}
    \begin{split}
    \Delta (\varphi_k(r)\Psi_k(\theta))
    = & \Psi_k\left(\varphi''_k+\frac{N-1}{r}\varphi'_k\right)+\frac{\varphi_k}{r^2}\Delta_{\mathbb{S}^{N-1}}\Psi_k \\
    = & \Psi_k\left(\varphi''_k+\frac{N-1}{r}\varphi'_k-\frac{\lambda_k}{r^2}\varphi_k\right).
    \end{split}
    \end{equation}
    Therefore, by standard regularity theory, the function $v$ is a solution of (\ref{Pwhlp}) if and only if $(\phi_k,\psi_k)\in \mathcal{C}\times \mathcal{C}$ is a classical solution of the system
    \begin{eqnarray}\label{p2c}
    \left\{ \arraycolsep=1.5pt
       \begin{array}{ll}
        \phi''_k+\frac{N-1}{r}\phi'_k-\frac{\lambda_k}{r^2}\phi_k+\frac{\psi_k}{r^{\alpha}}=0 \quad \mbox{in}\quad r\in(0,\infty),\\[3mm]
        \psi''_k+\frac{N-1}{r}\psi'_k-\frac{\lambda_k}{r^2}\psi_k+\frac{(p^*_{\alpha}-1)C_{N,\alpha}^{p^*_{\alpha}-2}}{r^{\alpha} (1+r^{2-\alpha})^4}\phi_k=0 \quad \mbox{in}\quad r\in(0,\infty),\\[3mm]
        \phi'_k(0)=\psi'_k(0)=0 \quad\mbox{if}\quad k=0,\quad \mbox{and}\quad \phi_k(0)=\psi_k(0)=0 \quad\mbox{if}\quad k\geq 1,
        \end{array}
    \right.
    \end{eqnarray}
    where $\mathcal{C}:=\{\omega\in C^2([0,\infty))| \int^\infty_0 r^\alpha |\omega''(r)+\frac{N-1}{r}\omega'(r)|^2 r^{N-1} dr<\infty\}$. Take the same variation as in the proof of Theorem \ref{thmbcm}, $|x|=r=s^q$ where $q=2/(2-\alpha)$ and let
    \begin{equation}\label{p2txy}
    X_k(s)=\phi_k(r),\quad Y_k(s)=q^2\psi_k(r),
    \end{equation}
    that transforms (\ref{p2c}) into the system
    \begin{eqnarray}\label{p2t}
    \left\{ \arraycolsep=1.5pt
       \begin{array}{ll}
        X''_k+\frac{M-1}{s}X'_k-\frac{\lambda_kq^2}{s^2}X_k+Y_k=0 \quad \mbox{in}\quad s\in(0,\infty),\\[3mm]
        Y''_k+\frac{M-1}{s}Y'_k-\frac{\lambda_kq^2}{s^2}Y_k+\frac{(M+4)(M-2)M(M+2)}{(1+s^2)^4}X_k=0 \quad \mbox{in}\quad s\in(0,\infty),\\[3mm]
        X'_k(0)=Y'_k(0)=0 \quad\mbox{if}\quad k=0,\quad \mbox{and}\quad X_k(0)=Y_k(0)=0 \quad\mbox{if}\quad k\geq 1,
        \end{array}
    \right.
    \end{eqnarray}
    in $(X_k,Y_k)\in \widetilde{\mathcal{C}}\times \widetilde{\mathcal{C}}$, where $\widetilde{\mathcal{C}}:=\{\omega\in C^2([0,\infty))| \int^\infty_0 |\omega''(s)+\frac{M-1}{s}\omega'(s)|^2 s^{M-1} ds<\infty\}$ and
    \begin{equation}
    M=\frac{2(N-\alpha)}{2-\alpha}>4.
    \end{equation}
    Here we have used the fact
    \begin{equation*}
    q^4(p^*_{\alpha}-1)C_{N,\alpha}^{p^*_{\alpha}-2}=\left[(M-4)(M-2)M(M+2)\right]\left(\frac{2M}{M-4}-1\right)=(M+4)(M-2)M(M+2).
    \end{equation*}

    Fix $M$, then let us now consider the following eigenvalue problem
    \begin{eqnarray}\label{p2te}
    \left\{ \arraycolsep=1.5pt
       \begin{array}{ll}
        X''+\frac{M-1}{s}X'-\frac{\mu}{s^2}X+Y=0 \quad \mbox{in}\quad s\in(0,\infty),\\[3mm]
        Y''+\frac{M-1}{s}Y'-\frac{\mu}{s^2}Y+\frac{(M+4)(M-2)M(M+2)}{(1+s^2)^4}X=0 \quad \mbox{in}\quad s\in(0,\infty).
        \end{array}
    \right.
    \end{eqnarray}
    When $M$ is an integer we can study (\ref{p2te}) as the linearized problem of the equation
    \begin{equation*}
    \Delta^2 U=(M-4)(M-2)M(M+2) U^{\frac{M+4}{M-4}},\quad U>0 \quad\mbox{in}\quad \mathbb{R}^M,
    \end{equation*}
    around the standard solution $U(x)=(1+|x|^2)^{-\frac{M-4}{2}}$ (note that we always have $M>4$). In this case, as in \cite[Theorem 2.2]{LW00}, we have that
    \begin{equation}\label{ptev}
    \mu_0=0; \quad \mu_1=M-1\quad \mbox{and}\quad X_0(s)=\frac{1-s^2}{(1+s^2)^{\frac{M-2}{2}}}; \quad X_1(s)=\frac{s}{(1+s^2)^{\frac{M-2}{2}}}.
    \end{equation}
    Moreover, even $M$ is not an integer we readily see that (\ref{ptev}) remains true. Therefore, we can conclude that (\ref{p2t}) has nontrivial solutions if and only if
    \begin{equation*}
    q^2\lambda_k\in \{0,M-1\},\quad \mbox{i.e.,}\quad \frac{4\lambda_k}{(2-\alpha)^2}\in \left\{0,\frac{2N-2-\alpha}{2-\alpha}\right\},
    \end{equation*}
    where $\lambda_k=k(N-2+k)$, $k\in\mathbb{N}$. If $4\lambda_k/(2-\alpha)^2=0$ then $k=0$. Moreover, if \[\frac{4\lambda_k}{(2-\alpha)^2}=\frac{2N-2-\alpha}{2-\alpha},\] then %$\left[\alpha+2(k-1)\right]\left[\alpha-2(N+k-1)\right]=0$,
    \begin{equation*}
    \left[\alpha+2(k-1)\right]\left[\alpha-2(N+k-1)\right]=0,
    \end{equation*}
    we obtain $\alpha=-2(k-1)$ since $4-N<\alpha<2$. Turning back to (\ref{p2c}) we obtain the solutions
    \begin{equation}\label{pyf}
    \phi_0(r)=\frac{1-r^{2-\alpha}}{(1+r^{2-\alpha})^{\frac{N-2}{2-\alpha}}} \quad\mbox{if}\quad \alpha\neq-2(k-1),\quad \forall k\in\mathbb{N}^+,
    \end{equation}
    and
    \begin{equation}\label{pye}
    \phi_0(r)=\frac{1-r^{2-\alpha}}{(1+r^{2-\alpha})^{\frac{N-2}{2-\alpha}}},\quad
    \phi_k(r)=\frac{r^k}{(1+r^{2-\alpha})^{\frac{N-2}{2-\alpha}}} \quad\mbox{if}\quad \alpha=-2(k-1),%\quad\mbox{for some}\quad k\in\mathbb{N}^+.
    \end{equation}
    for some $k\in\mathbb{N}^+$. That is, if $\alpha$ is not a negative even integer, then the space of solutions of (\ref{Pwhlp}) has dimension $1$ and is spanned by
    \begin{equation*}
    Z_0(x)=\frac{1-|x|^{2-\alpha}}{(1+|x|^{2-\alpha})^\frac{N-2}{2-\alpha}}.
    \end{equation*}
    If $\alpha=-2(k-1)$ for some $k\in\mathbb{N}^+$, then the space of solutions of (\ref{Pwhlp}) has dimension $1+\frac{(N+2k-2)(N+k-3)!}{(N-2)!k!}$ and is spanned by
    \begin{equation*}
    Z_0(x)=\frac{1-|x|^{2-\alpha}}{(1+|x|^{2-\alpha})^\frac{N-2}{2-\alpha}},\quad Z_{k,i}(x)=\frac{|x|^k\Psi_{k,i}(x)}{(1+|x|^{2-\alpha})^\frac{N-2}{2-\alpha}},
    \end{equation*}
    where $\{\Psi_{k,i}\}$, $i=1,\ldots,\frac{(N+2k-2)(N+k-3)!}{(N-2)!k!}$, form a basis of $\mathbb{Y}_k(\mathbb{R}^N)$, the space of all homogeneous harmonic polynomials of degree $k$ in $\mathbb{R}^N$. The proof of Theorem \ref{thmpwhl} is now complete.

    \qed

\section{{\bfseries Remainder terms of (CKN) inequality}}\label{sect:rt}

In this section, we consider the remainder terms of (CKN) inequality (\ref{Pi}) in radial space $D^{2,2}_{\alpha,rad}(\mathbb{R}^N)$ and give the proof of Theorem \ref{thmprt}.  We follow the arguments as those in \cite{BE91}, and also \cite{LW00}.

We define $u_\lambda(x):=\lambda^{\frac{N-4+\alpha}{2}}u(\lambda x)$ for all $\lambda>0$. Thus for simplicity of notations, we write $U_\lambda$ instead of $U_{\lambda,\alpha}$ and $S_\alpha$ instead of $S^{rad}(N,\alpha)$ if there is no possibility of confusion. Moreover, in order to shorten formulas we denote
    \begin{equation}\label{def:norm}
    \begin{split}
    \|\varphi\|: & =\|\varphi\|_{D^{2,2}_\alpha(\mathbb{R}^N)}=\left(\int_{\mathbb{R}^N}|x|^{\alpha}|\Delta \varphi|^2 dx\right)^{1/2}, \quad \mbox{for}\quad \varphi\in D^{2,2}_\alpha(\mathbb{R}^N),  \\
    \|\varphi\|_*: & =\|\varphi\|_{L^{p^*_\alpha}_\alpha(\mathbb{R}^N)}= \left(\int_{\mathbb{R}^N}|x|^{-\alpha}|\varphi|^{p^*_{\alpha}} dx\right)^{1/p^*_{\alpha}},\quad \mbox{for}\quad \varphi\in L^{p^*_\alpha}_\alpha(\mathbb{R}^N).
    \end{split}
    \end{equation}

    Consider the eigenvalue problem
    \begin{equation}\label{Pwhlep}
    \Delta(|x|^{\alpha}\Delta v)=\mu|x|^{-\alpha} U_{\lambda}^{p^*_{\alpha}-2}v \quad \mbox{in}\quad \mathbb{R}^N, \quad v\in D^{2,2}_\alpha(\mathbb{R}^N).
    \end{equation}
    By a simple scaling argument, we have that $\mu$ does not depending on $\lambda$. Moreover, from Theorem \ref{thmpwhl} we have:
    \begin{proposition}\label{propep}
    Assume $N\geq 3$, and let $4-N<\alpha<2$ be not a negative even integer. Let $\mu_i$, $i=1,2,\ldots,$ denote the eigenvalues of (\ref{Pwhlep}) in increasing order. Then $\mu_1=1$ is simple with eigenfunction $U_\lambda$ and $\mu_2=p^*_{\alpha}-1$ with the corresponding one-dimensional eigenfunction space spanned by $\{\frac{\partial U_\lambda}{\partial \lambda}\}$. Furthermore, eigenvalues do not depend on $\lambda$, and $\mu_3>\mu_2$.
    \end{proposition}

    The main ingredient in the proof of Theorem \ref{thmprt} is contained in the lemma below, where the behavior near  $\mathcal{M}_2=\{cU_{\lambda,\alpha}: c\in\mathbb{R}, \lambda>0\}$ is studied.

    \begin{lemma}\label{lemma:rtnm2b}
    Assume $N\geq 3$, and let $4-N<\alpha<2$ be not a negative even integer. Then for any sequence $\{u_n\}\subset D^{2,2}_{\alpha,rad}(\mathbb{R}^N)\backslash \mathcal{M}_2$ such that $\inf_n\|u_n\|>0$ and ${\rm dist}(u_n,\mathcal{M}_2)\to 0$, we have
    \begin{equation}\label{rtnmb}
    \lim\inf_{n\to\infty}\frac{\|u_n\|^2-S_\alpha\|u_n\|^2_*}{{\rm dist}(u_n,\mathcal{M}_2)^2}\geq 1-\frac{\mu_2}{\mu_3},
    \end{equation}
    where $\mu_2=p^*_\alpha-1<\mu_3$ are given as in Proposition \ref{propep}.
    \end{lemma}

    \begin{proof}
    Let $d_n:={\rm dist}(u_n,\mathcal{M}_2)=\inf_{c\in\mathbb{R}, \lambda>0}\|u_n-cU_\lambda\|\to 0$. We know that for each $u_n\in D^{2,2}_{\alpha,rad}(\mathbb{R}^N)$, there exist $c_n\in\mathbb{R}$ and $\lambda_n>0$ such that $d_n=\|u_n-c_nU_{\lambda_n}\|$. In fact,
    \begin{equation}\label{ikeda}
    \begin{split}
    \|u_n-cU_\lambda\|^2
    = & \|u_n\|^2+c^2\|U_\lambda\|^2-2c\langle u_n,U_\lambda\rangle_\alpha \\
    \geq & \|u_n\|^2+c^2\|U_1\|^2-2|c|\|u_n\| \|U_1\|.
    \end{split}
    \end{equation}
    Thus the minimizing sequence of $d_n^2$, say $\{c_{n,m},\lambda_{n,m}\}$, must satisfying $|c_{n,m}|\leq C$ which means $\{c_{n,m}\}$ is bounded.
    On the other hand,
    \begin{equation*}
    \begin{split}
    \left|\int_{|\lambda x|\leq \rho}|x|^{\alpha}\Delta u_n\Delta U_\lambda dx\right|
    \leq & \int_{|y|\leq \rho}|y|^{\alpha}|\Delta (u_n)_{\frac{1}{\lambda}}(y)||\Delta U_1(y)| dy \\
    \leq & \|u_n\|\left(\int_{|y|\leq \rho}|y|^{\alpha}|\Delta U_1|^2 dy\right)^{1/2} \\
    = & o_\rho(0)
    \end{split}
    \end{equation*}
    as $\rho\to 0$ which is uniform for $\lambda>0$, where $(u_n)_{\frac{1}{\lambda}}(y)=\lambda^{-\frac{N-4+\alpha}{2}}u_n(\lambda^{-1}y)$, and
    \begin{equation*}
    \begin{split}
    \left|\int_{|\lambda x|\geq \rho}|x|^{\alpha}\Delta u_n\Delta U_\lambda dx \right|
    \leq & \|U_1\|\left(\int_{|x|\geq \frac{\rho}{\lambda}}|x|^{\alpha}|\Delta u_n|^2 dy\right)^{1/2}
    =  o_\lambda(0)
    \end{split}
    \end{equation*}
    as $\lambda\to 0$ for any fixed $\rho>0$. By taking $\lambda\to 0$ and then $\rho\to 0$, we obtain
    \[\left|\int_{\mathbb{R}^N}|x|^{\alpha}\Delta u_n\Delta U_\lambda dx\right| \to 0\quad \mbox{as}\quad \lambda\to 0.\]
    %Note that $1+|\lambda x|\leq C'$ for $|\lambda x|\leq C$,
    Moreover, by the explicit form of $U_\lambda$ we have
    \begin{equation*}
    \begin{split}
    \left|\int_{|\lambda x|\leq R}|x|^{\alpha}\Delta u_n\Delta U_\lambda dx \right|
    \leq & \|U_1\|\left(\int_{| x|\leq \frac{R}{\lambda}}|x|^{\alpha}|\Delta u_n|^2 dx\right)^{1/2}
    = o_\lambda(0)
    \end{split}
    \end{equation*}
    as $\lambda\to +\infty$ for any fixed $R>0$ and
    \begin{equation*}
    \begin{split}
    \left|\int_{|\lambda x|\geq R}|x|^{\alpha}\Delta u_n\Delta U_\lambda dx\right|
    \leq & \int_{|y|\geq R}|y|^{\alpha}|\Delta (u_n)_{\frac{1}{\lambda}}(y)||\Delta U_1(y)| dy \\
    \leq & \|u_n\|\left(\int_{|y|\geq R}|y|^{\alpha}|\Delta U_1|^2 dy\right)^{1/2}
    =  o_R(0)
    \end{split}
    \end{equation*}
    as $R\to +\infty$ which is uniform for $\lambda>0$. Thus, by taking first $\lambda\to +\infty$ and then $R\to +\infty$, we also obtain
    \[\left|\int_{\mathbb{R}^N}|x|^{\alpha}\Delta u_n\Delta U_\lambda dx\right| \to 0\quad \mbox{as}\quad \lambda\to +\infty.\]
    It follows from (\ref{ikeda}) and $d_n\to 0$, $\inf_n\|u_n\|>0$ that the minimizing sequence $\{c_{n,m},\lambda_{n,m}\}$ must satisfying $|\lambda_{n,m}|\leq C$ which means $\{\lambda_{n,m}\}$ is bounded. Thus for each $u_n\in D^{2,2}_{\alpha,rad}(\mathbb{R}^N)$, $d_n^2$ can be attained by some $c_n\in\mathbb{R}$ and $\lambda_n>0$.

    Since $\mathcal{M}_2$ is two-dimensional manifold embedded in $D^{2,2}_{\alpha,rad}(\mathbb{R}^N)$, that is
    \[
    (c,\lambda)\in\mathbb{R}\times\mathbb{R}_+\to cU_\lambda\in D^{2,2}_{\alpha,rad}(\mathbb{R}^N),
    \]
    then the tangential space at $(c_n,\lambda_n)$ is given by
    \[
    T_{c_n U_{\lambda_n}}\mathcal{M}_2={\rm Span}\left\{U_{\lambda_n}, \frac{\partial U_\lambda}{\partial \lambda}\Big|_{\lambda=\lambda_n}\right\},
    \]
    and we must have that $(u_n-c_n U_{\lambda_n})$ is perpendicular to $T_{c_n U_{\lambda_n}}\mathcal{M}_2$. Proposition \ref{propep} implies that
    \begin{equation}\label{epkeyibbg}
    \mu_3\int_{\mathbb{R}^N}|x|^{-\alpha}U_{\lambda_n}^{p^*_{\alpha}-2}(u_n-c_n U_{\lambda_n})^2 \leq \int_{\mathbb{R}^N}|x|^\alpha |\Delta (u_n-c_n U_{\lambda_n})|^2.
    \end{equation}
    Let $u_n=c_n U_{\lambda_n}+d_n w_n$, then $w_n$ is perpendicular to $T_{c_n U_{\lambda_n}}\mathcal{M}_2$, $\|w_n\|=1$ and we can rewrite (\ref{epkeyibbg}) as follows:
    \begin{equation}\label{epkeyibbb}
    \int_{\mathbb{R}^N}|x|^{-\alpha}U_{\lambda_n}^{p^*_{\alpha}-2}w_n^2\leq \frac{1}{\mu_3}.
    \end{equation}
    Furthermore,
    \begin{equation*}
    \|u_n\|^2=d_n^2+c_n^2\|U_{\lambda_n}\|^2,
    \end{equation*}
    and by  using Taylor's expansion, it holds that
    \begin{equation}\label{epkeyiybb}
    \begin{split}
    \int_{\mathbb{R}^N}|x|^{-\alpha}|u_n|^{p^*_{\alpha}}
    = & \int_{\mathbb{R}^N}|x|^{-\alpha}|c_n U_{\lambda_n}+d_nw_n|^{p^*_{\alpha}} \\
    = & |c_n|^{p^*_{\alpha}}\int_{\mathbb{R}^N}|x|^{-\alpha}U_{\lambda_n}^{p^*_{\alpha}}
    +d_n p^*_{\alpha}|c_n|^{p^*_{\alpha}-1}\int_{\mathbb{R}^N}|x|^{-\alpha}U_{\lambda_n}^{p^*_{\alpha}-1}w_n  \\
    & +\frac{p^*_{\alpha}(p^*_{\alpha}-1)d_n^2  |c_n|^{p^*_{\alpha}-2} }{2}\int_{\mathbb{R}^N}|x|^{-\alpha}U_{\lambda_n}^{p^*_{\alpha}-2}w_n^2
    +o(d_n^2)  \\
    = & |c_n|^{p^*_{\alpha}}\int_{\mathbb{R}^N}|x|^{-\alpha}U_{\lambda_n}^{p^*_{\alpha}}  \\
    & + \frac{p^*_{\alpha}(p^*_{\alpha}-1)d_n^2  |c_n|^{p^*_{\alpha}-2} }{2} \int_{\mathbb{R}^N}|x|^{-\alpha}U_{\lambda_n}^{p^*_{\alpha}-2}w_n^2
    +o(d_n^2),
    \end{split}
    \end{equation}
    since
    \begin{equation*}
    \int_{\mathbb{R}^N}|x|^{-\alpha}U_{\lambda_n}^{p^*_{\alpha}-1}w_n=\int_{\mathbb{R}^N}|x|^\alpha \Delta U_{\lambda_n} \Delta w_n=0.
    \end{equation*}
    Then combining with (\ref{epkeyibbb}) and (\ref{epkeyiybb}), we obtain
    \begin{equation}\label{epkeyiyxbb}
    \begin{split}
    \|u_n\|^{2}_*
    \leq & \left(|c_n|^{p^*_{\alpha}}\|U_{\lambda_n}\|^{p^*_{\alpha}}_*+\frac{p^*_{\alpha}(p^*_{\alpha}-1)d_n^2  |c_n|^{p^*_{\alpha}-2} }{2\mu_3}
    +o(d_n^2)\right)^{\frac{2}{p^*_{\alpha}}} \\
    = & c_n^2\left(\|U_{\lambda_n}\|^{p^*_{\alpha}}_*+\frac{p^*_{\alpha}(p^*_{\alpha}-1)d_n^2  c_n^{-2}}{2\mu_3}
    +o(d_n^2)\right)^{\frac{2}{p^*_{\alpha}}}  \\
    = & c_n^2\left(\|U_{\lambda_n}\|^2_*+\frac{2}{p^*_{\alpha}}\frac{p^*_{\alpha}(p^*_{\alpha}-1)d_n^2  c_n^{-2}}{2\mu_3} \|U_{\lambda_n}\|^{2-p_\alpha^*}_*
    +o(d^2)\right) \\
    = & c_n^2\|U_{\lambda_n}\|^2_*+ \frac{d_n^2 (p^*_{\alpha}-1)}{\mu_3}\|U_{\lambda_n}\|^{2-p_\alpha^*}_*+o(d_n^2).
    \end{split}
    \end{equation}
    Therefore,
    \begin{equation}\label{epkeyiydzbb}
    \begin{split}
    \|u_n\|^2-S_\alpha\|u_n\|^{2}_*
    \geq & d_n^2+c_n^2\|U_{\lambda_n}\|^2- S_\alpha\left(c_n^2\|U_{\lambda_n}\|^2_*+ \frac{d_n^2 (p^*_{\alpha}-1)}{\mu_3}\|U_{\lambda_n}\|^{2-p_\alpha^*}_*+o(d_n^2)\right)  \\
    = & d_n^2 \left(1-\frac{p^*_\alpha-1}{\mu_3} S_\alpha \|U_{\lambda_n}\|_*^{2-p^*_\alpha}\right)
    +c_n^2(\|U_{\lambda_n}\|^2- S_\alpha\|U_{\lambda_n}\|^2_*)+o(d_n^2)  \\
    = & d_n^2\left(1-\frac{p^*_\alpha-1}{\mu_3}\right)+o(d_n^2),
    \end{split}
    \end{equation}
    which holds for all $n\in\mathbb{N}$ since $\|U\|^2=\|U\|_*^{p^*_\alpha}=S_\alpha^{p^*_\alpha/(p^*_\alpha-2)}$ and $\|U\|^2=S_\alpha\|U\|_*^2$ for all $U\in\mathcal{M}_2$, then (\ref{rtnmb}) follows immediately.
    \end{proof}

    \noindent{\bf Proof of Theorem \ref{thmprt}.} We argue by contradiction. In fact, if the theorem is false then there exists a sequence $\{u_n\}\subset D^{2,2}_{\alpha,rad}(\mathbb{R}^N)\backslash \mathcal{M}_2$ such that
    \begin{equation*}
    \frac{\|u_n\|^2-S_\alpha\|u_n\|^2_*}{{\rm dist}(u_n,\mathcal{M}_2)^2}\to 0,\quad \mbox{as}\quad n\to \infty.
    \end{equation*}
    By homogeneity, we can assume that $\|u_n\|=1$, and after selecting a subsequence we can assume that ${\rm dist}(u_n,\mathcal{M}_2)\to \xi\in[0,1]$ since ${\rm dist}(u_n,\mathcal{M}_2)=\inf_{c\in\mathbb{R}, \lambda>0}\|u_n-cU_{\lambda}\|\leq \|u_n\|$. If $\xi=0$, then we have a contradiction by Lemma \ref{lemma:rtnm2b}.

    The other possibility only is that $\xi>0$, that is
    \[{\rm dist}(u_n,\mathcal{M}_2)\to \xi>0\quad \mbox{as}\quad n\to \infty,\]
    then we must have
    \begin{equation}\label{wbsi}
    \|u_n\|^2-S_\alpha\|u_n\|^2_*\to 0,\quad \|u_n\|=1.
    \end{equation}
    Since $\{u_n\}\subset D^{2,2}_{\alpha,rad}(\mathbb{R}^N)\backslash \mathcal{M}_2$ are radial, making the changes that $v_n(s)=u_n(r)$ and $r=s^{2/(2-\alpha)}$, then (\ref{wbsi}) is equivalent to
    \begin{equation}\label{bsiy}
    \begin{split}
    \int^\infty_0\left[v_n''(s)+\frac{M-1}{s}v_n'(s)\right]^2 s^{M-1}ds
    -C(M)\left(\int^\infty_0|v_n(s)|^{\frac{2M}{M-4}}s^{M-1}ds\right)^{\frac{M-4}{M}}\to 0
    \end{split}
    \end{equation}
    where $M=\frac{2(N-\alpha)}{2-\alpha}>4$ and $C(M)=(M-4)(M-2)M(M+2)\left[\Gamma^2(\frac{M}{2})/(2\Gamma(M))\right]^{\frac{4}{M}}$, see the proof of Theorem \ref{thmbcm}. When $M$ is an integer  (\ref{bsiy}) is equivalent to
    \begin{equation}\label{bsib}
    \begin{split}
    \int_{\mathbb{R}^M}|\Delta v_n|^2 dx
    -S(M)\left(\int_{\mathbb{R}^M}|v_n|^{\frac{2M}{M-4}}dx\right)^{\frac{M-4}{M}}\to 0,\quad \|v_n\|_{D^{2,2}_{0}(\mathbb{R}^M)}=\left(\frac{2}{2-\alpha}\right)^{3/2},
    \end{split}
    \end{equation}
    where $S(M)=\pi^2(M-4)(M-2)M(M+2)\left[\Gamma(\frac{M}{2})/\Gamma(M)\right]^{\frac{4}{M}}$ is the best constant for the embedding of the space $D^{2,2}_0(\mathbb{R}^M)$ into $L^{2M/(M-4)}(\mathbb{R}^M)$, see \cite{Va93}. In this case, by Lions' concentration and compactness principle (see \cite[Theorem \uppercase\expandafter{\romannumeral 2}.4]{Li85-1}) as those in \cite{LW00}, we have that there exists a sequence of positive numbers $\lambda_n$ such that
    \begin{equation*}
    \lambda_n^{\frac{M-4}{2}}v_n(\lambda_n x)\to V\quad \mbox{in}\quad D^{2,2}_0(\mathbb{R}^M)\quad \mbox{as}\quad n\to \infty,
    \end{equation*}
    where $V(x)=c(a+|x|^2)^{-(M-4)/2}$ for some $c\neq 0$ and $a>0$, that is
    \begin{equation*}
    \tau_n^{\frac{N-4+\alpha}{2}}u_n(\tau_n x)\to U\quad \mbox{in}\quad D^{2,2}_\alpha(\mathbb{R}^N)\quad \mbox{as}\quad n\to \infty,
    \end{equation*}
    for some sequence $\{\tau_n\}$ and $U\in\mathcal{M}_2$, which implies
    \begin{equation*}
    {\rm dist}(u_n,\mathcal{M}_2)={\rm dist}\left(\tau_n^{\frac{N-4+\alpha}{2}}u_n(\tau_n x),\mathcal{M}_2\right)\to 0 \quad \mbox{as}\quad n\to \infty,
    \end{equation*}
    this is a contradiction. Moreover, even $M$ is not an integer we can also get analogous contradiction.
    \qed

\section{{\bfseries Finite-dimensional reduction}}\label{sectprp}

    In this section, we consider perturbation problem (\ref{Pwhp}) and give the proof of Theorem \ref{thmpwhp} by using Finite-dimensional reduction method. We always suppose that $4-N<\alpha<2$ and $\alpha$ is not a negative even integer.

    Given $h\in L^\infty(\mathbb{R}^N)\cap C(\mathbb{R}^N)$, we put
    \begin{equation}\label{defH}
    H[u]=\frac{1}{p^*_{\alpha}}\int_{\mathbb{R}^N}h(x)|x|^{-\alpha} u^{p^*_{\alpha}}_+dx.
    \end{equation}
    For $\varepsilon\in\mathbb{R}$ we introduce the perturbed energy functional $\mathcal{J}_\varepsilon$ and also the unperturbed energy functional $\mathcal{J}_0$ on $D^{2,2}_{\alpha,rad}(\mathbb{R}^N)$ given by
    \begin{equation*}
    \begin{split}
    \mathcal{J}_\varepsilon[u]=\mathcal{J}_0[u]-\varepsilon H[u]
    =\frac{1}{2}\int_{\mathbb{R}^N}|x|^{\alpha}|\Delta u|^2dx-\frac{1}{p^*_{\alpha}}\int_{\mathbb{R}^N}(1+\varepsilon h(x))|x|^{-\alpha} u^{p^*_{\alpha}}_+dx.
    \end{split}
    \end{equation*}
    Evidently, $\mathcal{J}_\varepsilon\in \mathcal{C}^2$ and any critical point $u$ of $\mathcal{J}_\varepsilon$ is a weak solution to
    \begin{equation*}
    \Delta(|x|^{\alpha}\Delta u)=(1+\varepsilon h(x))|x|^{-\alpha} u^{p^*_{\alpha}-1}_+.
    \end{equation*}
    If $u\neq 0$ and $|\varepsilon|\|h\|_\infty\leq 1$, then $u$ is positive by the strong maximum principle. Hence, $u$ solves (\ref{Pwhp}).

    Define now
    \begin{equation}\label{deful}
    \mathcal{U}:=\left\{U_{\lambda}(x)=\lambda^{\frac{N-4+\alpha}{2}}U_{1}(\lambda x)\big| \lambda>0\right\},
    \end{equation}
    and
    \begin{equation}\label{defri0}
    \mathcal{E}:=\left\{\omega\in D^{2,2}_{\alpha,rad}(\mathbb{R}^N): \langle\omega,\frac{\partial U_{\lambda}}{\partial \lambda}\rangle_\alpha=\int_{\mathbb{R}^N}|x|^{\alpha}\Delta \omega\Delta \frac{\partial U_{\lambda}}{\partial \lambda} dx=0\quad \mbox{for all}\quad \lambda>0 \right\}.
    \end{equation}
    For $\lambda>0$, we define the map $P_\lambda: D^{2,2}_\alpha(\mathbb{R}^N) \to D^{2,2}_\alpha(\mathbb{R}^N)$ by
    \begin{equation*}
    P_\lambda(u):=\lambda^{\frac{N-4+\alpha}{2}}u(\lambda x).
    \end{equation*}
    We can check that $P_\lambda$ converses the norms $\|\cdot\|$ and $\|\cdot\|_*$ (see the definitions as in (\ref{def:norm})), thus for every $\lambda>0$
    \begin{equation}\label{defplp}
    (P_\lambda)^{-1}=(P_\lambda)^t=P_{\lambda^{-1}}\quad \mbox{and}\quad \mathcal{J}_0=\mathcal{J}_0\circ P_\lambda,
    \end{equation}
    where $(P_\lambda)^t$ denotes the adjoint of $P_\lambda$. Twice differentiating the identity $\mathcal{J}_0=\mathcal{J}_0\circ P_\lambda$ yields for all $u, \varphi, \psi\in D^{2,2}_\alpha(\mathbb{R}^N)$
    \begin{equation*}
    (\mathcal{J}''_0[u]\phi,\psi)=(\mathcal{J}''_0[P_\lambda(u)]P_\lambda(\phi),P_\lambda(\psi)),
    \end{equation*}
    that is
    \begin{equation}\label{defpft}
    \mathcal{J}''_0[u]=(P_\lambda)^{-1}\circ\mathcal{J}''_0[P_\lambda(u)]\circ P_\lambda,\quad \forall u\in D^{2,2}_\alpha(\mathbb{R}^N).
    \end{equation}
    Differentiating (\ref{defplp}) we see that $P(\lambda,u):=P_\lambda(u)$ maps $(0,\infty)\times \mathcal{U}$ into $\mathcal{U}$, hence
    \begin{equation}\label{defpw}
    \frac{\partial P}{\partial u}(\lambda,u): T_u \mathcal{U}\to T_{P_\lambda(u)}\mathcal{U}\quad \mbox{and}\quad P_\lambda: (T_u \mathcal{U})^\perp\to (T_{P_\lambda(u)}\mathcal{U})^\perp.
    \end{equation}
    From Theorem \ref{thmpwhl}, we know that the manifold $\mathcal{U}$ is non-degenerate, then take the same argument as in \cite[Corollary 3.2]{FS03}, we can know that $\mathcal{J}''_0[U_1]$ is a self-adjoint Fredholm operator of index zero which maps the space $D^{2,2}_\alpha(\mathbb{R}^N)$ into $T_{U_1}\mathcal{U}^\perp$, and $\mathcal{J}''_0[U_1]\in \mathfrak{L}(T_{U_1}\mathcal{U}^\perp)$ is invertible. Consequently, using (\ref{defpft}) and (\ref{defpw}), we obtain in this case
    \begin{equation}\label{defpr}
    \|\mathcal{J}''_0[U_1]\|_{\mathfrak{L}(T_{U_1}\mathcal{U}^\perp)}=\|\mathcal{J}''_0[U]\|_{\mathfrak{L}(T_{U}\mathcal{U}^\perp)},\quad \forall U\in \mathcal{U}.
    \end{equation}

    \begin{lemma}\label{lemhw}
    Suppose that $h\in L^\infty(\mathbb{R}^N)\cap C(\mathbb{R}^N)$, then there exists a constant $C_1=C_1(\|h\|_\infty,\alpha,\lambda)>0$ such that for any $\lambda>0$ and for any $\omega\in D^{2,2}_{\alpha,rad}(\mathbb{R}^N)$,
    \begin{equation}\label{gh0}
    |H[U_{\lambda}+\omega]|\leq C_1(\||h|^{\frac{1}{p^*_\alpha}}U_{\lambda}\|^{p^*_\alpha}_*+\|\omega\|^{p^*_\alpha}),
    \end{equation}
    \begin{equation}\label{gh1}
    \|H'[U_{\lambda}+\omega]\|\leq C_1(\||h|^{\frac{1}{p^*_\alpha}}U_{\lambda}\|^{p^*_\alpha-1}_*+\|\omega\|^{p^*_\alpha-1}),
    \end{equation}
    \begin{equation}\label{gh2}
    \|H''[U_{\lambda}+\omega]\|\leq C_1(\||h|^{\frac{1}{p^*_\alpha}}U_{\lambda}\|^{p^*_\alpha-2}_*+\|\omega\|^{p^*_\alpha-2}).
    \end{equation}
    Moreover, if $\lim_{|x|\to 0}h(x)=\lim_{|x|\to \infty}h(x)=0$ then
    \begin{equation}\label{ghu}
    \||h|^{\frac{1}{p^*_\alpha}}U_{\lambda}\|_*\to 0 \quad\mbox{as}\quad \lambda\to 0\quad\mbox{or}\quad \lambda\to \infty.
    \end{equation}
    \end{lemma}

    \begin{proof}
    We will only show (\ref{gh2}) as (\ref{gh0})-(\ref{gh1}) follow analogously. By H\"{o}lder's inequality and since the embedding $D^{2,2}_{\alpha,rad}(\mathbb{R}^N)\hookrightarrow L^{p^*_{\alpha}}_{\alpha}(\mathbb{R}^N)$ is continuous, we have
    \begin{equation*}
    \begin{split}
    \|H''[U_{\lambda}+\omega]\|
    \leq & (p^*_\alpha-1)\sup_{\|g_1\|,\|g_2\|\leq 1} \int_{\mathbb{R}^N}\frac{|h(x)||U_{\lambda}+\omega|^{p^*_{\alpha}-2}|g_1||g_2|}{|x|^{\alpha}}dx\\
    \leq & (p^*_\alpha-1)\|h\|^{\frac{2}{p^*_\alpha}}_\infty \sup_{\|g_1\|,\|g_2\|\leq 1} \||h|^{\frac{1}{p^*_\alpha}}(U_{\lambda}+\omega)\|^{p^*_\alpha-2}_*\|g_1\|_*\|g_2\|_* \\
    \leq & c(\|h\|_\infty,\alpha,\lambda)\||h|^{\frac{1}{p^*_\alpha}}(U_{\lambda}+\omega)\|^{p^*_\alpha-2}_*.
    \end{split}
    \end{equation*}
    Then by using the triangle inequality and again $D^{2,2}_{\alpha,rad}(\mathbb{R}^N)\hookrightarrow L^{p^*_{\alpha}}_{\alpha}(\mathbb{R}^N)$, we can directly obtain (\ref{gh2}).

    Under the additional assumption $\lim_{|x|\to 0}h(x)=\lim_{|x|\to \infty}h(x)=0$, (\ref{ghu}) follows by the dominated convergence theorem and
    \begin{equation*}
    \int_{\mathbb{R}^N}\frac{|h(x)|U_{\lambda}^{p^*_{\alpha}}}{|x|^\alpha}dx=\int_{\mathbb{R}^N}\frac{|h(\lambda^{-1} x)|U_1^{p^*_{\alpha}}}{|x|^\alpha}dx.
    \end{equation*}
    \end{proof}

    In order to deal with the problem $\mathcal{J}'_\varepsilon[u]=0$ for $\varepsilon$ close to zero, we combine variational methods with the Lyapunov-Schmit reduction method, in the spirit of \cite{AGP99}, and also \cite{FS03,MN21}. The next lemma is the crucial step.
    \begin{lemma}\label{lemreg}
    Suppose that $h\in L^\infty(\mathbb{R}^N)\cap C(\mathbb{R}^N)$, then there exist constants $\varepsilon_0$, $C_2>0$ and a smooth function
    \begin{equation*}
    \omega=\omega(\lambda,\varepsilon):(0,\infty)\times(-\varepsilon_0,\varepsilon_0)\to D^{2,2}_{\alpha,rad}(\mathbb{R}^N)
    \end{equation*}
    such that for any $\lambda>0$ and $\varepsilon\in (-\varepsilon_0,\varepsilon_0)$,
    \begin{equation}\label{we}
    \omega(\lambda,\varepsilon)\in \mathcal{E},
    \end{equation}
    \begin{equation}\label{jes}
    \mathcal{J}'_\varepsilon[U_\lambda+\omega(\lambda,\varepsilon)]\eta=0,\quad \forall \eta\in \mathcal{E},
    \end{equation}
    \begin{equation}\label{wgx}
    \|\omega(\lambda,\varepsilon)\|\leq C_2|\varepsilon|.
    \end{equation}
    Moreover, if $\lim_{|x|\to 0}h(x)=\lim_{|x|\to \infty}h(x)=0$ then
    \begin{equation}\label{wgt0}
    \|\omega(\lambda,\varepsilon)\|\to 0 \quad\mbox{as}\quad \lambda\to 0\quad\mbox{or}\quad \lambda\to \infty,
    \end{equation}
    uniformly with respect to $\varepsilon$.
    \end{lemma}

    \begin{proof}
    Define $G: (0,\infty)\times D^{2,2}_{\alpha,rad}(\mathbb{R}^N)\times \mathbb{R}\times \mathbb{R}\to D^{2,2}_{\alpha,rad}(\mathbb{R}^N)\times \mathbb{R}$
    \begin{equation*}
    G(\lambda,\omega,l,\varepsilon):=(\mathcal{J}'_\varepsilon[U_\lambda+\omega]-l \dot{\xi}_\lambda,(\omega,\dot{\xi}_\lambda)),
    \end{equation*}
    where $\dot{\xi}_\lambda$ denotes the normalized tangent vector $\frac{d}{d\lambda}U_\lambda$.
    We observe
    \begin{equation}\label{fmt}
    \begin{split}
    \left(\left(\frac{\partial G}{\partial (\omega,l)}(\lambda,0,0,0)\right)(\omega,l), (\mathcal{J}''_0[U_\lambda]\omega-l \dot{\xi}_\lambda,(\omega,\dot{\xi}_\lambda))\right)=\|\mathcal{J}''_0[U_\lambda]\omega\|^2+l^2+|(\omega,\dot{\xi}_\lambda)|^2,
    \end{split}
    \end{equation}
    where
    \begin{equation*}
    \begin{split}
    \left(\frac{\partial G}{\partial (\omega,l)}(\lambda,0,0,0)\right)(\omega,l)= (\mathcal{J}''_0[U_\lambda]\omega-l \dot{\xi}_\lambda,(\omega,\dot{\xi}_\lambda)).
    \end{split}
    \end{equation*}
    From the invertibility of $\mathcal{J}''_0[U_1]$ we infer that $\frac{\partial G}{\partial (\omega,l)}(\lambda,0,0,0)$ is an injective Fredholm operator of index zero, hence invertible and by (\ref{defpr}) and (\ref{fmt}) we obtain
    \begin{equation}\label{gpn}
    \begin{split}
    \left\|\left(\frac{\partial G}{\partial (\omega,l)}(\lambda,0,0,0)\right)^{-1}\right\|
    \leq & \max\{1, \|(\mathcal{J}''_\varepsilon[U_\lambda])^{-1}\|\} \\
    = & \max\{1, \|(\mathcal{J}''_\varepsilon[U_1])^{-1}\|\}=:C_*.
    \end{split}
    \end{equation}
    If $G(\lambda,\omega,l,\varepsilon)=(0,0)$ for some $l\in\mathbb{R}$ then $\omega$ satisfies (\ref{we})-(\ref{jes}), and $G(\lambda,\omega,l,\varepsilon)=(0,0)$ if and only if $(\omega,l)=F_{\lambda,\varepsilon}(\omega,l)$, where
    \begin{equation*}
    F_{\lambda,\varepsilon}(\omega,l):=-\left(\frac{\partial G}{\partial (\omega,l)}(\lambda,0,0,0)\right)^{-1}G(\lambda,\omega,l,\varepsilon)+(\omega,l).
    \end{equation*}
    We will prove that $F_{\lambda,\varepsilon}$ is a contraction map in some ball $B_\rho$, where we may choose the radius $\rho=\rho(\varepsilon)>0$ independent of $U\in\mathcal{U}$. Here for $(\omega,l)\in B_\rho$, it means $\omega\in D^{2,2}_{\alpha,rad}(\mathbb{R}^N)$ and $l\in\mathbb{R}$ satisfying $\|(\omega,l)\|:=(\|\omega\|^2+l^2)^{1/2}\leq \rho$.

    Suppose $(\omega,l)\in B_\rho$. From (\ref{defpft}) and (\ref{gpn}), we can obtain that
    \begin{equation}\label{fnl}
    \begin{split}
    \left\|F_{\lambda,\varepsilon}(\omega,l)\right\|
    \leq & C_* \left\|G(\lambda,\omega,l,\varepsilon)-\left(\frac{\partial G}{\partial (\omega,l)}(\lambda,0,0,0)\right)(\omega,l)\right\| \\
    \leq & C_* \|\mathcal{J}'_\varepsilon[U_\lambda+\omega]-\mathcal{J}''_0[U_\lambda]\omega\| \\
    \leq & C_* \int^1_0\|\mathcal{J}''_0[U_\lambda+t\omega]-\mathcal{J}''_0[U_\lambda]\|dt \|\omega\|
    + C_*|\varepsilon| \|H'[U_{\lambda}+\omega]\| \\
    \leq & C_* \int^1_0\|\mathcal{J}''_0[U_1+tP_{\lambda^{-1}}(\omega)]-\mathcal{J}''_0[U_1]\|dt \|\omega\|
    + C_*|\varepsilon| \|H'[U_{\lambda}+\omega]\| \\
    \leq & C_* \rho \sup_{\|\omega\|\leq\rho} \|\mathcal{J}''_0[U_1+\omega]-\mathcal{J}''_0[U_1]\|
    + C_*|\varepsilon| \sup_{\|\omega\|\leq\rho}\|H'[U_{\lambda}+\omega]\|.
    \end{split}
    \end{equation}
    Analogously, for $(\omega_1,l_1), (\omega_2,l_2)\in B_\rho$ we get
    \begin{equation}\label{fnpl}
    \begin{split}
    \frac{\left\|F_{\lambda,\varepsilon}(\omega_1,l_1)-F_{\lambda,\varepsilon}(\omega_2,l_2)\right\|}{C_*\|\omega_1-\omega_2\|}
    \leq & \frac{\|\mathcal{J}'_\varepsilon[U_\lambda+\omega_1]-\mathcal{J}'_\varepsilon[U_\lambda+\omega_2]-\mathcal{J}''_0[U_\lambda](\omega_1-\omega_2)\|}{\|\omega_1-\omega_2\|} \\
    \leq & \int^1_0\|\mathcal{J}''_\varepsilon[U_\lambda+\omega_2+t(\omega_1-\omega_2)]-\mathcal{J}''_0[U_\lambda]\|dt  \\
    \leq & \int^1_0\|\mathcal{J}''_0[U_\lambda+\omega_2+t(\omega_1-\omega_2)]-\mathcal{J}''_0[U_\lambda]\|dt  \\
    & + |\varepsilon| \int^1_0\|H''[U_{\lambda}+\omega_2+t(\omega_1-\omega_2)]\|dt \\
    \leq & \sup_{\|\omega\|\leq 3\rho} \|\mathcal{J}''_0[U_1+\omega]-\mathcal{J}''_0[U_1]\| \\
    &+ |\varepsilon| \sup_{\|\omega\|\leq 3\rho}\|H''[U_{\lambda}+\omega]\|.
    \end{split}
    \end{equation}
    We may choose $\rho_0>0$ such that
    \begin{equation*}
    C_*\sup_{\|\omega\|\leq 3\rho_0} \|\mathcal{J}''_0[U_1+\omega]-\mathcal{J}''_0[U_1]\|\leq\frac{1}{2},
    \end{equation*}
    and $\varepsilon_0>0$ such that
    \begin{equation*}
    \varepsilon_0 C_*\sup_{U\in \mathcal{U}, \|\omega\|\leq 3\rho_0}\|H''[U+\omega]\| <\frac{1}{3}
    \quad \mbox{and}\quad
    \varepsilon_0 C_* \sup_{U\in \mathcal{U}, \|\omega\|\leq \rho_0}\|H'[U+\omega]\| \leq \frac{\rho_0}{2}.
    \end{equation*}
    With these choices and the above estimates, it is easy to see that for every $U_\lambda\in \mathcal{U}$ and $|\varepsilon|\leq \varepsilon_0$,  $F_{\lambda,\varepsilon}$ maps $B_{\rho_0}$ into itself and is a contraction map. Therefore, $F_{\lambda,\varepsilon}$ has a unique fixed point $(\omega(\lambda,\varepsilon),l(\lambda,\varepsilon))$ in $B_{\rho_0}$ and it is a consequence of the implicit function theorem that $\omega$ and $l$ are continuously differentiable.

    From (\ref{fnl}) we also infer that $F_{\lambda,\varepsilon}$ maps $B_{\rho}$ into $B_{\rho}$, whenever
    \begin{equation*}
    2 C_*|\varepsilon| \sup_{U\in \mathcal{U}, \|\omega\|\leq\rho_0}\|H'[U+\omega]\|\leq \rho\leq\rho_0.
    \end{equation*}
    In order to get (\ref{wgx}), here we take
    \begin{equation*}
    \rho=\rho(\varepsilon):=2 C_*|\varepsilon| \sup_{U\in \mathcal{U}, \|\omega\|\leq\rho_0}\|H'[U+\omega]\|,
    \end{equation*}
    consequently, due to the uniqueness of the fixed point we have
    \begin{equation*}
    \|\omega(\lambda,\varepsilon),l(\lambda,\varepsilon)\|\leq 2 C_*|\varepsilon| \sup_{U\in \mathcal{U}, \|\omega\|\leq\rho_0}\|H'[U+\omega]\|,
    \end{equation*}
    which gives (\ref{wgx}).

    Let us now prove (\ref{wgt0}). Set
    \begin{equation}\label{defrl}
    \rho_\lambda=\min\left\{4 \varepsilon_0 C_* C_1 \||h|^{\frac{1}{p^*_\alpha}}U_{\lambda}\|^{p^*_\alpha-1}_*, \rho_0, (8 \varepsilon_0 C_* C_1)^{2-p^*_\alpha}\right\}
    \end{equation}
    where $C_1=C_1(\|h\|_\infty,\alpha,\lambda)>0$ is given as in Lemma \ref{lemhw}. In view of (\ref{gh1}), for any $|\varepsilon|<\varepsilon_0$ and $\lambda>0$ we have that
    \begin{equation*}
    C_*|\varepsilon| \sup_{\|\omega\|\leq\rho_\lambda}\|H'[U_\lambda+\omega]\|
    \leq |\varepsilon| C_* C_1 \||h|^{\frac{1}{p^*_\alpha}}U_{\lambda}\|^{p^*_\alpha-1}_*+ |\varepsilon| C_* C_1 \rho^{p^*_\alpha-2}_\lambda \rho_\lambda.
    \end{equation*}
    Since $\rho^{p^*_\alpha-2}_\lambda\leq (8 \varepsilon_0 C_* C_1)^{-1}$, we have
    \begin{equation*}
    C_*|\varepsilon| \sup_{\|\omega\|\leq\rho_\lambda}\|H'[U_\lambda+\omega]\|
    < |\varepsilon| C_* C_1 \||h|^{\frac{1}{p^*_\alpha}}U_{\lambda}\|^{p^*_\alpha-1}_*+ \frac{1}{4}\rho_\lambda \leq \frac{1}{2}\rho_\lambda,
    \end{equation*}
    then by the above argument, we can conclude that $F_{\lambda,\varepsilon}$ maps $B_{\rho_\lambda}$ into $B_{\rho_\lambda}$. Consequently, due to the uniqueness of the fixed-point we have
    \begin{equation*}
    \|\omega(\lambda,\varepsilon)\|\leq \rho_\lambda.
    \end{equation*}
    From (\ref{ghu}) and (\ref{defrl}), we have that $\rho_\lambda \to 0$ as $\lambda\to 0$ or $\lambda\to \infty$, then we get (\ref{wgt0}).
    \end{proof}

    Under the assumptions of Lemma \ref{lemreg}, for $|\varepsilon|<\varepsilon_0$ we may define
    \begin{equation}\label{defule}
    \mathcal{U^\varepsilon}:=\left\{u\in D^{2,2}_{\alpha,rad}(\mathbb{R}^N)|u=U_{\lambda}+\omega(\lambda,\varepsilon),\quad \lambda\in(0,\infty)\right\},
    \end{equation}
    where $\omega(\lambda,\varepsilon)\in \mathcal{E}$ is given as in Lemma \ref{lemreg}. Note that $\mathcal{U^\varepsilon}$ is a one-dimensional manifold. The next lemma will show that finding critical points for functional can be reduced to a finite dimensional problem.

    \begin{lemma}\label{lemcuve}
    Under the assumptions of Lemma \ref{lemreg}, we may choose $\varepsilon_0>0$ such that for every $|\varepsilon|<\varepsilon_0$ the manifold $\mathcal{U^\varepsilon}$ is a natural constraint for $\mathcal{J}_\varepsilon$, i.e., every critical point of $\mathcal{J}_\varepsilon|_{\mathcal{U^\varepsilon}}$ is the critical point of $\mathcal{J}_\varepsilon$.
    \end{lemma}

    \begin{proof}
    Fix $u\in \mathcal{U^\varepsilon}$ such that $\mathcal{J}'_\varepsilon|_{\mathcal{U^\varepsilon}}[u]=0$. From the definition of $\mathcal{U^\varepsilon}$ and by Lemma \ref{lemreg}, we can know the form of $u$ is that $u=U_{\lambda}+\omega(\lambda,\varepsilon)$ for some $\lambda>0$. In the following, we use a dot for the derivation with respect to $\lambda$. From the definition of $\mathcal{E}$, it holds that $\langle \dot{U}_\lambda,\omega(\lambda,\varepsilon) \rangle_\alpha=0$ for all $\lambda>0$, then we obtain
    \begin{equation}\label{cuvef}
    \langle \ddot{U}_\lambda,\omega(\lambda,\varepsilon) \rangle_\alpha+\langle \dot{U}_\lambda,\dot{\omega}(\lambda,\varepsilon) \rangle_\alpha=0.
    \end{equation}
    Moreover differentiating the identity $U_\lambda=P_{\sigma}U_{\lambda/\sigma}$ with respect to $\lambda$ we obtain
    \begin{equation}\label{cuvefs}
    \dot{U}_\sigma:=\frac{\partial U_\lambda}{\partial \lambda}\Big|_{\lambda=\sigma}=\frac{1}{\sigma}P_{\sigma}\dot{U}_{1}\quad \mbox{and}\quad \ddot{U}_\sigma:=\frac{\partial^2 U_\lambda}{\partial \lambda^2}\Big|_{\lambda=\sigma}=\frac{1}{\sigma^2}P_{\sigma}\ddot{U}_{1}.
    \end{equation}
    From (\ref{jes}) we get $\mathcal{J}'_\varepsilon[u]=c_1 \dot{U}_\lambda$ for some $\lambda>0$ and $c_1\in\mathbb{R}$. Then from (\ref{cuvef})-(\ref{cuvefs}) and (\ref{wgx}), we obtain
    \begin{equation*}
    \begin{split}
    0
    = & \mathcal{J}'_\varepsilon[u](\dot{U}_\lambda+\dot{\omega}(\lambda,\varepsilon))= c_1\langle\dot{U}_\lambda,\dot{U}_\lambda+\dot{\omega}(\lambda,\varepsilon)\rangle_\alpha
    =c_1 \lambda^{-2}(\|\dot{U}_1\|^2-\langle P_{\lambda}\ddot{U}_1, \omega(\lambda,\varepsilon)\rangle_\alpha) \\
    = & c_1 \lambda^{-2}(\|\dot{U}_1\|^2-\langle\ddot{U}_1, P_{\lambda^{-1}}\omega(\lambda,\varepsilon)\rangle_\alpha)
    =c_1\lambda^{-2}(\|\dot{U}_1\|^2-\|\ddot{U}_1\|O(1)\varepsilon).
    \end{split}
    \end{equation*}
    Finally, we see that for small $\varepsilon>0$ the number $c_1$ must be zero, therefore the conclusion follows.
    \end{proof}

    Now, we are in position to prove the main result.

    \noindent{\bf Proof of Theorem \ref{thmpwhp}.} Choose $\varepsilon>0$ small, then let $u^\varepsilon_\lambda=U_\lambda+\omega(\lambda,\varepsilon)$, where $\omega(\lambda,\varepsilon)$ is given in Lemma \ref{lemreg} and write
    \begin{equation*}
    \mathcal{J}_\varepsilon[u^\varepsilon_\lambda]=\mathcal{J}_0[U_{\lambda}]+\frac{1}{2}(\|u^\varepsilon_\lambda\|^2-\|U_{\lambda}\|^2)
    -\frac{1}{p^*_{\alpha}}\int_{\mathbb{R}^N}\frac{(1+\varepsilon h(x))((u^\varepsilon_\lambda)^{p^*_{\alpha}}_+ -U_{\lambda}^{p^*_{\alpha}})}{|x|^{\alpha} }dx
    -\varepsilon H[U_{\lambda}].
    \end{equation*}
    Recall that $\|U_{\lambda}\|=\|U_{1}\|$ does not depend on $\lambda$, then from (\ref{wgt0}) we infer that $u^\varepsilon_\lambda$ is uniformly bounded in $D^{2,2}_\alpha(\mathbb{R}^N)$, and $\|u^\varepsilon_\lambda-U_{\lambda}\|=\|\omega(\lambda,\varepsilon)\|=o(1)$ as $\lambda\to 0$ or $\lambda\to \infty$, therefore
    \begin{equation*}
    |\|u^\varepsilon_\lambda\|^2-\|U_{\lambda}\|^2|\leq (\|u^\varepsilon_\lambda\|+\|U_{\lambda}\|)\|u^\varepsilon_\lambda-U_{\lambda}\|=o(1).
    \end{equation*}
    Moreover, by H\"{o}lder's inequality and (CKN) inequality (\ref{Pi}) we obtain
    \begin{equation*}
    \begin{split}
    \left|\int_{\mathbb{R}^N}\frac{(1+\varepsilon h(x))((u^\varepsilon_\lambda)^{p^*_{\alpha}}_+ -U_{\lambda}^{p^*_{\alpha}})}{|x|^{\alpha} }dx\right|
    \leq & C\int_{\mathbb{R}^N}\frac{((u^\varepsilon_\lambda)^{p^*_{\alpha}-1}_+ +U_{\lambda}^{p^*_{\alpha}-1})|(u^\varepsilon_\lambda)_+ -U_{\lambda}|}{|x|^{\alpha}}dx \\
    \leq & C\|u^\varepsilon_\lambda-U_{\lambda}\|=o(1).
    \end{split}
    \end{equation*}
    Finally, from (\ref{gh0}), (\ref{ghu}) and (\ref{wgt0}), we already noticed that $|H[U_{\lambda}]|=o(1)$ as $\lambda\to 0$ or $\lambda\to \infty$, then we can conclude that
    \begin{equation*}
    \Gamma_{\varepsilon}(\lambda):=\mathcal{J}_\varepsilon[u^\varepsilon_\lambda]=\mathcal{J}_0[U_{\lambda}]+o(1)=\mathcal{J}_0[U_{1}]+o(1),\quad \mbox{as}\quad \lambda\to 0\quad \mbox{or} \quad \lambda\to\infty,
    \end{equation*}
    that is,
    \begin{equation*}
    \lim_{\lambda\to 0}\Gamma_{\varepsilon}(\lambda)=\lim_{\lambda\to \infty}\Gamma_{\varepsilon}(\lambda)=\mathcal{J}_0[U_1],
    \end{equation*}
    uniformly with respect to $\varepsilon$.
    Thus, $\Gamma_{\varepsilon}$ has at least one critical point $\lambda_{\varepsilon}$ (in fact, $\Gamma_{\varepsilon}$ might be constant and in this case we obtain infinitely many critical points). Hence $u^\varepsilon_{\lambda_\varepsilon}$ is a critical point for $\mathcal{J}_\varepsilon$ by Lemma \ref{lemcuve}, and the proof of Theorem \ref{thmpwhp} is now complete.

    \qed

\noindent{\bfseries Acknowledgements}

The research has been supported by National Natural Science Foundation of China 11971392, Natural Science Foundation of Chongqing, China cstc2019jcyjjqX0022 and Fundamental Research Funds for the Central Universities XDJK2019TY001.
We are very grateful to the referee for insightful suggestions, which has led to an important improvement of the paper.

    \end{document}